\documentclass{amsproc}
\usepackage{amsmath}
\usepackage{enumerate}
\usepackage{amsmath,amsthm,amscd,amssymb}

\usepackage{latexsym}
\usepackage{upref}
\usepackage{verbatim}
\usepackage{pstricks}
\usepackage[mathscr]{eucal}
\usepackage{dsfont}

\usepackage{graphicx}
\usepackage[colorlinks,hyperindex]{hyperref}

\usepackage{hhline}

\usepackage{booktabs}

\usepackage[OT2,OT1]{fontenc}
\newcommand\cyr
{
\renewcommand\rmdefault{wncyr}
\renewcommand\sfdefault{wncyss}
\renewcommand\encodingdefault{OT2}
\normalfont
\selectfont
}
\DeclareTextFontCommand{\textcyr}{\cyr}

\newtheorem{theorem}{Theorem}
\newtheorem{lemma}[theorem]{Lemma}
\newtheorem{remark}[theorem]{Remark}

\theoremstyle{definition}
\newtheorem{definition}{Definition}

\theoremstyle{plain}
\newtheorem{proposition}[definition]{Proposition}
\newtheorem{corollary}[definition]{Corollary}

\theoremstyle{remark}
\newtheorem{example}[definition]{Example}

\chardef\bslash=`\\ 

\newcommand{\ran}{\text{\rm{Ran}}}

\newcommand{\calD}{{\mathcal D}}

\newcommand{\calH}{{\mathcal H}}

\newcommand{\calS}{{\mathcal S}}

\renewcommand{\Im}{\text{\rm Im}}

\def\sM{{\mathfrak M}}

      \def\dC{{\mathbb C}}

      \def\dR{{\mathbb R}}

\def\cD{{\mathcal D}}

\def\cS{{\mathcal S}}

\def\RE{{\rm Re\,}}

\DeclareMathOperator{\IM}{Im}

\newcommand{\eval}[2][\right]{\relax
  \ifx#1\right\relax \left.\fi#2#1\rvert}

\begin{document}

\title{Dissipation and c-Entropy in Nevanlinna-Pick Interpolation}

\author{S. Belyi*}
\address{Department of Mathematics\\ Troy University\\
Troy, AL 36082, USA\\
}
\curraddr{}
\email{sbelyi@troy.edu}
\thanks{\textit{*Corresponding author}: Sergey Belyi}


\author[Makarov]{K. A. Makarov}
\address{Department of Mathematics\\
 University of Missouri\\
  Columbia, MO 63211, USA}
\email{makarovk@missouri.edu}


\author{E. Tsekanovskii}
\address{Department of Mathematics, Niagara University,  Lewiston,
NY  14109, USA} \email{\tt tsekanov@niagara.edu}


\date{DD/MM/2004}



\begin{abstract}
We study interpolation L-systems realizing finite Nevanlinna-Pick data sets and analyze
their structural and quantitative characteristics. Explicit formulas are derived for the
c-entropy and dissipation coefficient, two intrinsic invariants that describe the
dissipative structure of interpolation L-systems. These quantities depend only on the
geometric placement of interpolation nodes in $\dC_+$, attaining maximal finite values
for purely imaginary nodes. The interpolation model $\Theta_\Delta$ and its unitary
equivalents reveal that these invariants form a direct link between analytic
interpolation data and the dynamical properties of L-systems. Particular attention is
given to symmetric configurations, where the impedance function admits an explicit
rational representation and a natural physical interpretation in terms of equivalent
$LC$-networks.

\end{abstract}

\maketitle

\tableofcontents


\section{Introduction}\label{s1}

The theory of L-systems provides a functional model framework for dissipative operators and their transfer and impedance characteristics. The formal definition of  L-systems and their elements are presented in Section \ref{s2}. The main theme of this paper is  the study L-systems whose impedance functions realize finite Nevanlinna-Pick interpolation data (see \cite{AlTs1}, \cite{ABT}). We focus on two intrinsic invariants -- the \textit{c-entropy} and the \textit{dissipation coefficient} (introduced and discussed in  \cite{BT-21, BT-22, BT-23, BT-16}) -- which capture the internal dissipative behavior and depend only on the geometry of the interpolation nodes in $\mathbb{C}_+$.

In the framework of the approach  suggested in  \cite{ABT}
we derive explicit formulas for these invariants and characterize when the c-entropy is finite, maximal, or infinite.
As an application, we show that symmetric node configurations admit a physical interpretation in terms of equivalent electrical networks of coupled LC-branches, while degenerate data correspond to purely capacitive elements. This establishes a direct link between analytic properties of interpolation L-systems and their underlying physical and energetic structure.

\smallskip
The paper is organized as follows.

Section~\ref{s2} recalls basic concepts and definitions concerning L-systems, c-entropy, and the dissipation coefficient in the scalar setting.

Section~\ref{s3} introduces the interpolation L-system $\Theta_\Delta$ associated with a finite Nevanlinna-Pick data set and establishes its key analytic properties.

Sections~\ref{s4} and~\ref{s5} present the one-dimensional multiplication-operator model and the coupling construction of L-systems, establishing explicit formulas for their transfer and impedance functions.

Section~\ref{s6} evaluates the c-entropy and dissipation coefficients for these L-systems and identifies conditions under which the entropy is finite, maximal, or infinite.

Section~\ref{s7} focuses on symmetric interpolation configurations, presenting explicit impedance formulas, connections with rational Herglotz-Nevanlinna functions, and their physical interpretation as equivalent $LC$-networks.

Finally, Section~\ref{s8} contains concluding remarks summarizing the main results and their analytical and physical implications.

\section{Preliminaries}\label{s2}

We begin by recalling several basic notions from the theory of L-systems that will be used throughout the paper.  Canonical L-systems provide an operator-theoretic framework for modeling linear stationary conservative and dissipative systems in a Hilbert space $\calH$ (see, e.g., \cite{Bro}, \cite{ABT}, \cite{BT-21}, \cite{Lv2} ).

For a pair of Hilbert spaces $\calH_1$, $\calH_2$ denote by $[\calH_1,\calH_2]$ the set of all bounded linear operators from $\calH_1$ to $\calH_2$.

Let $T$ be a bounded linear operator in a Hilbert space $\calH$, $K\in[E,\calH]$, and $J$ be a bounded, self-adjoint, and unitary operator in $E$, where $E$ is another Hilbert space with $\dim E<\infty$. Let also
\begin{equation}\label{e5-3}
\IM T=KJK^*.
\end{equation}

By definition, the  \textbf{Liv\v{s}ic canonical system } or simply \textbf{canonical L-system} is just the  array
\begin{equation}\label{BL}
 \Theta =
\left(%
\begin{array}{ccc}
  T & K & J \\
  \calH &  & E \\
\end{array}%
\right).
 \end{equation}
 The spaces $\calH$ and $E$ here  are
called \textit{state} and
\textit{input-output} spaces, and the operators
$T$, $K$, $J$ will be refered to as \textit{main},
\textit{channel}, and \textit{directing} operators, respectively.

Notice that relation \eqref{e5-3} implies
\begin{equation*}
    \ran (\IM T)\subseteq\ran(K).
\end{equation*}

We  associate with an L-system $\Theta$ two  analytic functions,  the \textbf{transfer  function} of the L-system $\Theta$
\begin{equation}\label{e6-3-3}
W_\Theta (z)=I-2iK^\ast (T-zI)^{-1}KJ,\quad z\in \rho (T),
\end{equation}
and also the \textbf{impedance function}  given by the formula
\begin{equation}\label{e6-3-5}
V_\Theta (z) = K^\ast (\RE T - zI)^{-1} K, \quad z\in  \rho (\RE T).
\end{equation}

 The transfer function $W_\Theta (z)$ of the L-system $\Theta $ and function $V_\Theta (z)$ of the form (\ref{e6-3-5}) are connected by the following relations valid for $\IM z\ne0$, $z\in\rho(T)$,
\begin{equation}\label{e6-3-6}
\begin{aligned}
V_\Theta (z) &= i [W_\Theta (z) + I]^{-1} [W_\Theta (z) - I]J,\\
W_\Theta(z)&=(I+iV_\Theta(z)J)^{-1}(I-iV_\Theta(z)J).
\end{aligned}
\end{equation}

The canonical L-system $\Theta$ in \eqref{BL} is called \emph{minimal} if the operator $T$ is \emph{prime} (completely non-self-adjoint) in $\mathcal H$, i.e., there exists no nontrivial reducing invariant subspace of $\mathcal H$ on which $T$ induces a self-adjoint operator.

We say that a canonical L-system
$$
\Theta_1=\left(%
\begin{array}{ccc}
  T_1 & K_1 & J \\
  \calH_1 &  & E \\
\end{array}%
\right), $$ is \textbf{unitarily equivalent} to an L-system
$$ \Theta_2=\left(%
\begin{array}{ccc}
  T_2 & K_2 & J \\
  \calH_2 &  & E \\
\end{array}%
\right),
$$
if there exists an isometric mapping $U$ of the space $\calH_1$ onto
$\calH_2$ such that
\begin{equation}\label{e5-16}
    UT_1=T_2U,\quad UK_1=K_2.
\end{equation}
It was shown in \cite{ABT} (see also \cite{Bro}) that  if  canonical L-systems
$\Theta_1$ and $\Theta_2$ are unitarily equivalent, then the set $\rho(T_1)$ of regular points of $T_1$ coincides with the set $\rho(T_2)$ of regular points of $T_2$, and $W_{\Theta_1}( z)=W_{\Theta_2}( z)$, ($ z\in\rho(T_1)$). Conversely, if  canonical L-systems $\Theta_1$ and $\Theta_2$ are \emph{minimal}  and $W_{\Theta_1}(z)=W_{\Theta_2}(z)$ in a neighborhood of infinity, then $\Theta_1$ and $\Theta_2$ are unitarily equivalent.

In general (without minimality), the equality of transfer functions determines the system uniquely only up to a self-adjoint direct summand; equivalently, the \emph{principal (minimal) parts} of $\Theta_1$ and $\Theta_2$ are unitarily equivalent (see \cite{ABT}, \cite{Bro}).

In the general theory of L-systems, the state-space operator $T$ acts in a Hilbert space $\calH$, the channel space $E$ is finite-dimensional, and the signature operator $J$ in $E$ determines the L-system's input-output structure
(see, e.g., \cite{Bro}, \cite{BT-21}).  In the present paper, however, we restrict our attention to the simplest nontrivial class of L-systems of the form~\eqref{BL}, for which $E=\dC$ and $J=1$, or
\begin{equation}\label{BL-1}
 \Theta =
\left(%
\begin{array}{ccc}
  T & K & 1 \\
  \calH &  & \dC \\
\end{array}%
\right).
 \end{equation}
This scalar setting suffices for the interpolation problems considered here and allows all relevant quantities (transfer and impedance functions, c-entropy, and dissipation coefficient) to be treated as scalar functions while retaining the essential operator-theoretic structure. In particular \eqref{e6-3-6} takes an easier form of scalar Caley transforms
\begin{equation}\label{e-8-WV}
V_\Theta (z) = i \frac{W_\Theta (z) - 1}{W_\Theta (z) + 1},\quad
W_\Theta(z)=\frac{1-iV_\Theta(z)}{1+iV_\Theta(z)}.
\end{equation}

Within this scalar framework, we define the \emph{c-entropy} of an L-system $\Theta$ of the form~\eqref{BL-1} by the following relation. This definition  was introduced in \cite{BT-16} and further developed in \cite{BT-21}, \cite{BT-22}, \cite{BT-23}, \cite{MTBook}.
\begin{definition}\label{d-9-ent}
Let $\Theta$ be an L-system of the form \eqref{BL-1}. The quantity
\begin{equation}\label{e-80-entropy-def}
    \calS=-\ln (|W_\Theta(-i)|),
\end{equation}
where $W_\Theta(z)$ is the transfer function of $\Theta$, is called the \textbf{coupling entropy} (or \textbf{c-entropy}) of the L-system $\Theta$.
\end{definition}

 Note that if, in addition,  the point $z=i$ belongs to $\rho(T)$, then we also have that
\begin{equation*}
     \calS=\ln (|W_\Theta(i)|).
\end{equation*}
Let us also recall the definition of the dissipation coefficient of an L-system.
\begin{definition}[{cf. \cite{BT-16}}, \cite{BT-21}]\label{d-10}
Let $\Theta$ be an L-system  of the form \eqref{BL-1} with c-entropy $\calS$. Then
the quantity
 \begin{equation}\label{e-69-ent-dis}
\calD=1-e^{-2\calS}
\end{equation}
 is called the \textbf{coefficient of dissipation} (or dissipation coefficient) of the L-system $\Theta$.
\end{definition}
Although the dissipation coefficient was introduced above via its relation to the c-entropy, it admits an equivalent representation that expresses $\calD$ directly in terms of the main operator $T$ and the channel operator $K$ of the L-system.
\begin{lemma}\label{l-1-D-TK}
Let $\Theta$  be a scalar L-system of the form \eqref{BL-1} with bounded main operator $T$ and channel operator $K$ satisfying $\IM T = K K^*$. Then the dissipation coefficient $\calD$ of $\Theta$ admits the representation
\begin{equation}\label{e-12-D-TK}
\calD=1-|W_\Theta(-i)|^2=4\,K^*(T+iI)^{-1}(T^*-iI)^{-1}K
=4\|(T+iI)^{-1}K\|^2.
\end{equation}
\end{lemma}
\begin{proof}
The first part of \eqref{e-12-D-TK} follows immediately from \eqref{e-80-entropy-def} and \eqref{e-69-ent-dis}.
 Furthermore, recall that the transfer function of $\Theta$ here is given by \eqref{e6-3-3} or
$$
W_\Theta(z)=1-2i\,K^*(T-zI)^{-1}K.
$$
A direct computation (see also \cite{ABT}, \cite{Bro}) yields the identity
$$
1-W_\Theta(z)\overline{W_\Theta(w)}=2i\,(z-\overline{w})\,K^*(T-zI)^{-1}(T^*-\overline{w}I)^{-1}K.
$$
Setting $z=w=-i$ we obtain
$$
1-|W_\Theta(-i)|^2=4\,K^*(T+iI)^{-1}(T^*-iI)^{-1}K.
$$
Finally, since $T^*-iI=(T+iI)^*$, the last expression equals $4\|(T+iI)^{-1}K\|^2$, which completes the proof.
\end{proof}
Formula \eqref{e-12-D-TK} shows that the dissipation coefficient can be expressed directly in terms of the main operator $T$ and the channel direction $K$, without explicit reference to the c-entropy.

The notions introduced in this section (particularly the scalar formulation of L-systems of the form~\eqref{BL-1} and the associated invariants of c-entropy and dissipation coefficient) provide the analytical foundation for the interpolation models developed in the sequel.  In the next section, we construct an explicit interpolation L-system in the Pick form governed by the interpolation data $\{(z_\ell,v_\ell)\}$.

\section{System Interpolation and L-systems in the Pick Form}\label{s3}

We now recall the L-system formulation of the Nevanlinna-Pick interpolation problem and the following general definition of interpolation systems found in \cite{ABT}. Let $E$ be a finite-dimensional Hilbert space and let $\{z_\ell, v_\ell\}$ denote the interpolation data as described below.

\begin{definition}[\cite{ABT}]\label{Def3}
Let $\Theta$ be the Liv\v{s}ic canonical system of the form \eqref{BL}
and $V_\Theta(z)$ its impedance function defined by \eqref{e6-3-3},\eqref{e6-3-6}.
Then $\Theta$ is called an \textbf{interpolation system} (solution) in the
Nevanlinna-Pick interpolation problem for the data
$\{z_\ell \in \dC_+,\, \ell=1,\dots,m\}$ and
$\{v_\ell \in [E,E],\, \IM v_\ell \ge 0,\, \ell=1,\dots,m\}$ if
\[
   V_\Theta(z_\ell) = v_\ell, \qquad \ell=1,\dots,m.
\]
\end{definition}


For the remainder of this paper we, however,  restrict our attention to the scalar case $E=\mathbb{C}$ and L-systems of the form \eqref{BL-1}.  In this setting, each interpolation value $v_\ell$ is a complex number with $\IM v_\ell > 0$, so that the interpolation data $\{(z_\ell, v_\ell)\}$ consist of points in the upper half-plane $\dC_+$.   All subsequent constructions and results will therefore be formulated for $E=\mathbb{C}$.

Following \cite{ABT}, \cite{AlTs1} we set
\begin{equation}\label{PQ}
{\mathbb P}=\left(\frac{v_k-\bar v_j}{z_k-\bar
z_j}\right)_{k,j=1,\ldots m}, \qquad {\mathbb Q}=
\left(\frac{z_kv_k-\bar z_j\bar v_j}{z_k-\bar z_j}
\right)_{k,j=1,\ldots m}.
\end{equation}

We now summarize the structural results on interpolation systems obtained in \cite{ABT}.
Suppose $\{z_k\}_{k=1}^m \subset \dC_+$ and $\{v_k\}_{k=1}^m \subset \dC_+$ are interpolation data
for which the Pick matrix $\mathbb{P}$ in \eqref{PQ} is strictly positive.\footnote{A matrix is called
strictly positive if it is non-negative and invertible.} In this case, there exists a
scattering interpolation system $\Theta$ with state space of dimension $m$ that solves the
interpolation problem for this data.

\medskip
\noindent
\textbf{Construction of the system.}
Consider the space $\dC^m$ with inner product
\[
   (\xi,\eta)_{\dC_{\mathbb{P}}^m} = \eta^* \mathbb{P}\xi.
\]
The operator $\vec A = \mathbb{P}^{-1}\mathbb{Q}$ is self-adjoint with respect to this inner product.
Let
\[
   \vec g = \mathbb{P}^{-1}\varphi,
   \qquad \varphi = \begin{pmatrix}  v_1 \\ \vdots \\  v_m \end{pmatrix},
\]
and define
\begin{equation}\label{49a}
   \vec T = \vec A + i(\cdot,\vec g)_{\dC_{\mathbb{P}}^m}\,\vec g,
   \qquad
   \vec K c = c\cdot \vec g.
\end{equation}
Then ${\rm Im}\,\vec T = \vec K \vec K^*$, and the Liv\v{s}ic canonical scattering system
\[
   \vec\Theta =
   \begin{pmatrix}
      \vec T & \vec K & 1 \\
      \dC_{\mathbb{P}}^m & & \dC
   \end{pmatrix}
\]
is well defined. The associated impedance function has the form
\begin{equation}\label{50}
   V_{\vec\Theta}(z) = \vec K^*(\RE \vec T - zI)^{-1}\vec K.
\end{equation}
A direct calculation shows that $V_{\vec\Theta}(z_k)=v_k$ for $k=1,\dots,m$.
Indeed, since
$\RE\vec T=\vec A=\mathbb{P}^{-1}\mathbb{Q}$ and $\vec g=\mathbb{P}^{-1}\varphi$,
we have
\begin{align*}
V_{\vec\Theta}(z)
&= \left ((\mathbb{P}^{-1}\mathbb{Q}-zI)^{-1}\vec g,\vec g\right)_{\dC_{\mathbb{P}}^m}
= \left (\mathbb{P}(\mathbb{P}^{-1}\mathbb{Q}-zI)^{-1}\vec g,\vec g\right)_{\dC^m}
\\
&= \left (\mathbb{P}(\mathbb{P}^{-1}\mathbb{Q}-zI)^{-1}\mathbb{P}^{-1}\varphi,\mathbb{P}^{-1}\varphi\right)_{\dC^m}
= \left ((\mathbb{P}^{-1}\mathbb{Q}-zI)^{-1}\mathbb{P}^{-1}\varphi,\varphi\right)_{\dC^m}
\\
&= \left ((\mathbb{Q}-z\mathbb{P})^{-1}\varphi,\varphi\right)_{\dC^m}
= \varphi^*(\mathbb{Q}-z\mathbb{P})^{-1}\varphi.
\end{align*}
Using the definitions of $\mathbb{P}$ and $\mathbb{Q}$, one verifies that
\[
(\mathbb{Q}-z_k\mathbb{P})_{kj}=\frac{z_kv_k-\bar z_j\bar v_j}{z_k-\bar z_j}-z_k\frac{v_k-\bar v_j}{z_k-\bar z_j}=\bar v_j,\quad j=1,\dots,m,
\]
so the $k$-th row of $\mathbb{Q}-z_k\mathbb{P}$ equals
$(\bar v_1,\dots,\bar v_m)=\varphi^*$. Hence
\[
e_k^{\,T}(\mathbb{Q}-z_k\mathbb{P})=\varphi^*,
\]
and therefore
\[
\varphi^*(\mathbb{Q}-z_k\mathbb{P})^{-1}=e_k^{\,T},
\]
where $e_k$ denotes the column vector in $\dC^m$ whose $k$-th entry is $1$ and all other entries are $0$. Therefore,
\[
V_{\vec\Theta}(z_k)
= \left ((\mathbb{Q}-z_k\mathbb{P})^{-1}\varphi,\varphi\right)_{\dC^m}
=\varphi^*(\mathbb{Q}-z_k\mathbb{P})^{-1}\varphi
= e_k^{\,T}\varphi
= v_k,
\]
as claimed.

Thus $\vec\Theta$ is indeed a solution of the Nevanlinna-Pick interpolation problem.
We will refer to this realization as the \textbf{Liv\v{s}ic system in the Pick form}.

Recall that a bounded operator $T$ in a Hilbert space is called \emph{dissipative} if
\[
   \IM (Tx,x) \;\geq\; 0, \qquad x \in \mathcal{H}.
\]
For the operator $\vec T$ defined in \eqref{49a}, one computes
\[
   \IM \vec T = \tfrac{1}{2i}\,(\vec T - \vec T^*)
              = (\cdot,\vec g)_{\dC_{\mathbb{P}}^m}\,\vec g
              = \vec K \vec K^* .
\]
Since $\vec K \vec K^*$ is positive semidefinite, we obtain for every $x \in \dC_{\mathbb{P}}^m$
\[
   (\IM \vec T\, x, x) = (\vec K \vec K^* x, x) = \|\vec K^* x\|^2 \;\geq\; 0.
\]
Therefore $\vec T$ is dissipative. 

It is shown in \cite{ABT} that if the canonical scattering system of the form \eqref{BL-1}
\[
   \Theta =
   \begin{pmatrix}
      T & K & 1 \\
      \mathcal H & & \dC
   \end{pmatrix},
   \qquad {\rm Im}\,T = KK^*,
\]
is an interpolation system for the same data with $\dim \mathcal H = m$, provided
\begin{equation}\label{62}
   \operatorname{cls}\{ (T - z_k I)^{-1}K\dC : k=1,\dots,m\} = \mathcal H,
\end{equation}
then $\mathbb{P}$ is strictly positive and $\Theta$ is unitarily equivalent to
the system $\vec\Theta$ in the Pick form.

Moreover, when $\mathbb{P}>0$, the interpolation values corresponding to eigenvalues of $\vec T$ are necessarily equal to $i$. More precisely, suppose that $z_{k_1},\dots,z_{k_p}$ is a (nonempty) subset of the interpolation nodes that are eigenvalues of $\vec T$. Then
\[
v_{k_1}=\cdots=v_{k_p}=i.
\]
Conversely, if for some indices $k_1,\dots,k_p$ the interpolation data satisfy
\[
v_{k_1}=\cdots=v_{k_p}=i,
\]
then the corresponding nodes $z_{k_1},\dots,z_{k_p}$ are eigenvalues of $\vec T$.

This correspondence follows from the Cayley transform relation between the impedance and transfer functions (see \eqref{e-8-WV}), since the condition $V_{\vec\Theta}(z_k)=i$ is equivalent to a pole of $W_{\vec\Theta}(z)$ at $z=z_k$, and hence to $z_k$ being an eigenvalue of $\vec T$. A detailed proof can be found in \cite[Theorems 11.6.3--11.6.4]{ABT}.

The preceding characterization highlights the distinguished role of the value $i$. We now turn to the special case in which all interpolation values satisfy $v_k=i$, where the corresponding interpolation system admits an explicit model realization.

\medskip
\noindent
\textbf{Explicit model for $v_k=i$.}
The case when all interpolation values equal $i$ admits a simple model realization.
Define
\begin{equation}\label{99}
   T_\Delta =
   \begin{pmatrix}
      z_1 & 2i\sqrt{\IM z_1\cdot\IM z_2} & \cdots & 2i\sqrt{\IM z_1\cdot\IM z_m} \\
      0   & z_2 & \cdots & 2i\sqrt{\IM z_2\cdot\IM z_m} \\
      \vdots & \vdots & \ddots & \vdots \\
      0 & 0 & \cdots & z_m
   \end{pmatrix},
\end{equation}
so that $\IM T_\Delta = (\cdot,g_\Delta)g_\Delta$ with
\[
   g_\Delta = \begin{pmatrix}\sqrt{\IM z_1}\\ \sqrt{\IM z_2}\\ \vdots \\ \sqrt{\IM z_m}\end{pmatrix}.
\]
It was shown in \cite[Theorem 11.6.5]{ABT} that the corresponding L-system
\begin{equation}\label{e-27-Theta}
   \Theta_\Delta =
   \begin{pmatrix}
      T_\Delta & K_\Delta & 1 \\
      \dC^m & & \dC
   \end{pmatrix},
   \qquad K_\Delta(c)=c\cdot g_\Delta,
\end{equation}
has transfer function
\begin{equation}\label{e-28=W}
   W_{\Theta_\Delta}(z) = \prod_{k=1}^m \frac{z-\bar z_k}{z-z_k},
\end{equation}
and impedance function
\[
   V_{\Theta_\Delta}(z) = i\frac{W_{\Theta_\Delta}(z)-1}{W_{\Theta_\Delta}(z)+1}.
\]
Clearly $V_{\Theta_\Delta}(z_k)=i$ for $k=1,\dots,m$. Substituting \eqref{e-28=W} into the Cayley transform relation above we obtain the following explicit formula for the impedance function:
\begin{equation}\label{e-29-Vz}
V_{\Theta_\Delta}(z)
= i\,\frac{\displaystyle
\prod_{k=1}^{m}(z-\bar z_k)
   - \prod_{k=1}^{m}(z-z_k)}
{\displaystyle
\prod_{k=1}^{m}(z-\bar z_k)
   + \prod_{k=1}^{m}(z-z_k)},
\qquad z \in \dC_+.
\end{equation}
It follows from the construction and from \cite[Chapter~11]{ABT} that the interpolation L-system $\Theta_\Delta$ is minimal.

Furthermore, for any other interpolation system $\Theta$ with the same data and $\dim \mathcal H=m$, the transfer functions coincide, and $\Theta$ is unitarily equivalent to $\Theta_\Delta$ (see \cite[Theorem 11.6.5]{ABT}). Thus the solution is unique up to unitary equivalence.

In a summary, we have described the interpolation L-systems in the Pick form, following \cite{ABT}.
For strictly positive Pick data, one obtains a canonical scattering L-system $\vec\Theta$ whose
state-space operator $\vec T$ is dissipative and uniquely determined up to unitary equivalence.
The spectral properties of $\vec T$ reflect the interpolation values: if a node $z_k$ is an
eigenvalue of $\vec T$, then the corresponding value must be $v_k=i$, and conversely, whenever
$v_k=i$ the point $z_k$ belongs to the spectrum of $\vec T$.
In the special case when all interpolation values equal $i$, the system admits an explicit model
realization $\Theta_\Delta$, and every other solution is unitarily equivalent to it.

These general interpolation results provide the foundation for the explicit constructions developed in the next section.
In particular, we now turn to the study of L-systems associated with multiplication operators, which will serve as the building blocks for the subsequent analysis of coupling, c-entropy, and dissipation coefficients.

\section{Multiplication Operators and Coupling of L-systems}\label{s4}

In this section we present, in prose, the one-dimensional multiplication-operator model and the coupling of such canonical L-systems, recording the explicit formulas for the main and channel operators as well as the resulting transfer and impedance functions. We also state the multiplicativity of transfer functions under coupling. The material is somewhat standard (see \cite{Bro,BMkT-2,BT-21,BT-23,MT10}) and will be used later in our discussion of c-entropy and the dissipation coefficient.

\medskip
\noindent\textbf{Multiplication operator model.}
Let $\calH$ be a one-dimensional Hilbert space with normalized basis vector $h_0\in\calH$.
For a fixed $\lambda_0\in\dC$ with $\IM \lambda_0 >0$, define the multiplication operator
\begin{equation}\label{mult-T}
   Th = \lambda_0 h, \qquad h\in\calH.
\end{equation}
Its adjoint is $T^*h = \bar\lambda_0 h$, so that
\[
   \IM T = (\IM \lambda_0)I, \qquad \RE T = (\RE \lambda_0)I.
\]
Introducing the channel operator
\begin{equation}\label{mult-K}
   Kc = (\sqrt{\IM\lambda_0}\,c)h_0, \qquad c\in\dC,
\end{equation}
one checks that $KK^* = (\IM\lambda_0)I = \IM T$.
Thus the L-system
\begin{equation}\label{mult-Theta}
   \Theta =
   \begin{pmatrix}
      T & K & 1 \\
      \calH & & \dC
   \end{pmatrix}
\end{equation}
is well defined. Since $\dim\calH=1$, the multiplication model \eqref{mult-Theta}
is automatically minimal. Its transfer and impedance functions take the simple closed forms
\begin{equation}\label{mult-W}
   W_\Theta(z) = 1 - 2i\frac{\IM\lambda_0}{\lambda_0 - z}
               = \frac{\bar\lambda_0 - z}{\lambda_0 - z},
\end{equation}
\begin{equation}\label{mult-V}
   V_\Theta(z) = \frac{\IM\lambda_0}{\RE\lambda_0 - z}.
\end{equation}
By inspection, $V_\Theta(z)$ is a Herglotz-Nevanlinna function.
This one-dimensional multiplication model provides the simplest nontrivial
example of a dissipative L-system and will serve as a building block
for more complicated constructions.

\medskip
\noindent\textbf{Coupling of multiplication operator L-systems.}
Now suppose we have two such systems associated with
multiplication operators
\begin{equation}\label{mult-T12}
   T_1 h = \lambda_0 h, \qquad T_2 h = \mu_0 h, \qquad h\in\calH,
\end{equation}
with corresponding L-systems $\Theta_1$ and $\Theta_2$ of the form \eqref{mult-Theta}.
Following \cite{BT-23}, their coupling is defined by
\begin{equation}\label{coupling}
   \Theta =
   \begin{pmatrix}
      \mathbf{T} & \mathbf{K} & 1 \\
      \calH\oplus\calH & & \dC
   \end{pmatrix},
\end{equation}
where
\begin{equation}\label{T-K-coupling}
   \mathbf{T} = \begin{pmatrix} T_1 & 2iK_1K_2^* \\ 0 & T_2 \end{pmatrix},
   \qquad
   \mathbf{K} = \begin{pmatrix} K_1 \\ K_2 \end{pmatrix}.
\end{equation}
A direct computation shows that $\IM \mathbf{T} = \mathbf{K}\mathbf{K}^*$,
so $\Theta$ is again a canonical L-system.
The transfer function of the coupled system satisfies the multiplicative identity
\begin{equation}\label{mult-theorem}
   W_{\Theta}(z) = W_{\Theta_1}(z)\cdot W_{\Theta_2}(z),
\end{equation}
which is the essence of the so-called Multiplication Theorem (see also Appendix~A).

The corresponding impedance function admits the explicit representation
\begin{equation}\label{V-coupling}
   V_{\Theta}(z) = \frac{\IM(\lambda_0+\mu_0)z - \IM(\lambda_0\mu_0)}
                        {\RE(\lambda_0+\mu_0)z - \RE(\lambda_0\mu_0) - z^2}.
\end{equation}
This formula follows from substituting \eqref{mult-W} into
$V_\Theta(z)=i\frac{W_\Theta(z)-1}{W_\Theta(z)+1}$ and simplifying.

\begin{lemma}[Minimality under finite coupling]\label{lem:coupling-minimal}
Let $\Theta_1,\dots,\Theta_n$ be minimal canonical L-systems of the form
\eqref{mult-Theta} acting in one-dimensional Hilbert spaces, and let
$\Theta=\Theta_1\cdot\Theta_2\cdots\Theta_n$ denote their (iterated) coupling.
Then the coupled L-system $\Theta$ is minimal.
\end{lemma}

\begin{proof}
We first prove the claim for $n=2$.
Let $\Theta_1$ and $\Theta_2$ be as in the statement and let $\Theta$ be their coupling.
Minimality of a canonical L-system
$$
\Theta=
\begin{pmatrix}
T & K & 1\\
\mathcal H&&\dC
\end{pmatrix}
$$
means
\[
\operatorname{cls}\{(T-zI)^{-1}K\dC:\ z\in\rho(T)\}=\mathcal H .
\]
Since $\dim\calH_1=\dim\calH_2=1$, each $\Theta_j$ is minimal.

For the coupling $\Theta$ acting in $\calH_1\oplus\calH_2$ we have
\[
\mathbf T=
\begin{pmatrix}
T_1&2iK_1K_2^*\\
0&T_2
\end{pmatrix},
\qquad
\mathbf K=
\begin{pmatrix}
K_1\\K_2
\end{pmatrix}.
\]
For $z\in\rho(\mathbf T)$ the resolvent is upper triangular,
\[
(\mathbf T-zI)^{-1}=
\begin{pmatrix}
(T_1-zI)^{-1}&*\\
0&(T_2-zI)^{-1}
\end{pmatrix},
\]
hence
\[
(\mathbf T-zI)^{-1}\mathbf K c
=
\begin{pmatrix}
(T_1-zI)^{-1}K_1c+*\\
(T_2-zI)^{-1}K_2c
\end{pmatrix}.
\]
Varying $c\in\dC$ and $z\in\rho(\mathbf T)$ shows that the cyclic subspace generated by $(\mathbf T-zI)^{-1}\mathbf K\dC$ contains both $\calH_1$ and $\calH_2$. Therefore $\Theta$ is minimal.

For general $n$, set $\Theta^{(1)}=\Theta_1$ and define iteratively
$\Theta^{(k)}=\Theta^{(k-1)}\cdot\Theta_k$.
By the already proved two-factor case, if $\Theta^{(k-1)}$ is minimal then so is
$\Theta^{(k)}$.
Since $\Theta^{(1)}$ is minimal, induction yields minimality of $\Theta^{(n)}=\Theta$.
\end{proof}

\begin{remark}[Invariance under unitary equivalence]
\label{r-4}
If two L-systems $\Theta_1,\Theta_2$ of the form \eqref{BL} are unitarily equivalent, then (see \cite{ABT}) $W_{\Theta_1}(z)\equiv W_{\Theta_2}(z)$ and, consequently
\[
   \cS(\Theta_1)=\cS(\Theta_2), \qquad \cD(\Theta_1)=\cD(\Theta_2).
\]
Hence $\cS$ and $\cD$ are intrinsic invariants of the interpolation problem, independent
of the choice of realization. 
\end{remark}

The coupling construction preserves the canonical form of the L-system and multiplies transfer functions as in \eqref{mult-theorem}. Together with the continuity of Donoghue data classes established in \cite{BT-23}, this framework provides the theoretical foundation for the introduction of c-entropy and dissipation coefficients in the next section. For completeness, an independent proof of the Multiplication Theorem is included in Appendix~\ref{A}.

\section{c-Entropy and DC of an L-system with multiplication operator}\label{s5}

Having introduced the definitions of c-entropy and dissipation coefficient in Section \ref{s2}, and developed in Section~\ref{s4} the explicit models of multiplication-operator L-systems together with their couplings, we are now in a position to evaluate these invariants concretely. The results presented below follow \cite{BT-23} and provide closed-form expressions that will serve as the basis for later applications.

In this section we evaluate explicitly the c-entropy and the dissipation coefficient for L-systems associated with multiplication operators, and then extend these computations to their couplings. The relevant definitions of c-entropy and dissipation coefficient were introduced in Section~2, and here we apply them to obtain closed-form expressions. The results summarized below follow \cite{BT-23}.

\medskip
\noindent\textbf{Single multiplication operator.}
Consider an L-system $\Theta$ of the form \eqref{mult-Theta} built from a multiplication operator $T h=\lambda_0 h$ with $\IM \lambda_0>0$. Its transfer function \eqref{mult-W} can be used to compute the c-entropy directly. Indeed,
$$
   |W_\Theta(-i)| = \left|\frac{\bar\lambda_0+i}{\lambda_0+i}\right|
   = \sqrt{\frac{(\RE\lambda_0)^2+(1-\IM\lambda_0)^2}{(\RE\lambda_0)^2+(1+\IM\lambda_0)^2}}.
$$
By definition, $\cS=-\ln|W_\Theta(-i)|$, which simplifies to
\begin{equation}\label{entropy-single}
   \cS = \tfrac{1}{2}\ln\!\left(
          \frac{(\RE\lambda_0)^2+(1+\IM\lambda_0)^2}
               {(\RE\lambda_0)^2+(1-\IM\lambda_0)^2}
        \right).
\end{equation}

The dissipation coefficient $\cD$ of $\Theta$ is obtained from $\cD=1-e^{-2\cS}$, yielding
\begin{equation}\label{diss-single}
   \cD = \frac{4\,\IM\lambda_0}{(\RE\lambda_0)^2+(1+\IM\lambda_0)^2}.
\end{equation}
Thus, for the one-dimensional multiplication model both invariants are available in closed form in terms of the real and imaginary parts of $\lambda_0$.

\medskip
\noindent\textbf{Coupling of two multiplication models.}
Now let $\Theta$ be the coupling of two systems $\Theta_1$ and $\Theta_2$ of the form \eqref{mult-Theta}, with defining parameters $\lambda_0$ and $\mu_0$ respectively. From the multiplicativity of transfer functions \eqref{mult-theorem}, we have
\[
   W_\Theta(-i) = W_{\Theta_1}(-i)\,W_{\Theta_2}(-i),
\]
and therefore
\[
   \cS = -\ln|W_\Theta(-i)| = -\ln|W_{\Theta_1}(-i)| - \ln|W_{\Theta_2}(-i)|
       = \cS_1 + \cS_2.
\]
Hence the c-entropy is additive under coupling, justifying the terminology \emph{coupling entropy}. If either $\cS_1$ or $\cS_2$ is infinite, then so is $\cS$.

In terms of $\lambda_0$ and $\mu_0$, the c-entropy of the coupled system can be written as
\begin{equation}\label{entropy-coupling}
   \cS = \tfrac{1}{2}\ln\!\left(
          \frac{(\RE\lambda_0)^2+(1+\IM\lambda_0)^2}{(\RE\lambda_0)^2+(1-\IM\lambda_0)^2}
        \cdot
          \frac{(\RE\mu_0)^2+(1+\IM\mu_0)^2}{(\RE\mu_0)^2+(1-\IM\mu_0)^2}
        \right).
\end{equation}

\medskip
\noindent\textbf{Dissipation coefficient under coupling.}
If $\cD$, $\cD_1$, and $\cD_2$ denote the dissipation coefficients of $\Theta$, $\Theta_1$, and $\Theta_2$ respectively, then the identity $\cS=\cS_1+\cS_2$ combined with $\cD=1-e^{-2\cS}$ gives
\[
   1-\cD = (1-\cD_1)(1-\cD_2),
\]
or equivalently
\begin{equation}\label{diss-coupling-general}
   \cD = \cD_1+\cD_2-\cD_1\cD_2.
\end{equation}
Thus, the dissipation coefficient of the coupling is determined algebraically by the coefficients of its components.

Finally, substituting the explicit formulas \eqref{diss-single} for $\cD_1$ and $\cD_2$ yields a closed form in terms of $\lambda_0$ and $\mu_0$:
\begin{equation}\label{diss-coupling}
   \cD = \frac{4\,\IM\lambda_0\,(|\mu_0|^2+1) + 4\,\IM\mu_0\,(|\lambda_0|^2+1)}
              {\big[(\RE\lambda_0)^2+(1+\IM\lambda_0)^2\big]\,
               \big[(\RE\mu_0)^2+(1+\IM\mu_0)^2\big]}.
\end{equation}

\begin{remark}[Coupling versus equivalence]
\label{r-5}
The interpolation system $\Theta_\Delta$ of the form \eqref{e-27-Theta} constructed in Section~\ref{s3} is not literally a coupling of one-dimensional multiplication models. Its construction is canonical and tied
to the interpolation data. However, its transfer function
\[
   W_{\Theta_\Delta}(z)=\prod_{k=1}^m \frac{z-\bar z_k}{z-z_k}
\]
coincides with that of the $m$ coupled multiplication-operator L-systems obtained from parameters $\lambda_0=z_k$.
Thus $\Theta_\Delta$ is unitarily equivalent to such a coupling, since both realizations are minimal (see Section~\ref{s3} and Lemma~\ref{lem:coupling-minimal}) and have identical transfer functions.

\end{remark}

In a summary for this section we note that for multiplication-operator L-systems, the c-entropy and dissipation coefficient admit explicit formulas \eqref{entropy-single} and \eqref{diss-single}. Under coupling, c-entropy is additive \eqref{entropy-coupling} and the dissipation coefficient satisfies the nonlinear law \eqref{diss-coupling-general}, with the closed form \eqref{diss-coupling} available in terms of the defining parameters. These explicit expressions form the basis for the quantitative analysis of L-system couplings in the sequel.

\section{Interpolation model $\Theta_\Delta$ and its invariants}\label{s6}

We now arrive at the main result of this paper. In Section~\ref{s3} we introduced the explicit
interpolation system $\Theta_\Delta$ given in \eqref{e-27-Theta}, with transfer function
\eqref{e-28=W}. In Section~\ref{s4} we developed the multiplication operator model and its
coupling, and in Section~\ref{s5} we obtained closed formulas for the c-entropy and
dissipation coefficient in this setting. Bringing these elements together, we can evaluate
the invariants of $\Theta_\Delta$.

The following theorem establishes a closed formula for the c-entropy of the interpolation
system.

\begin{theorem}\label{t-6}
Let $\Theta_\Delta$ be the interpolation L-system of the form \eqref{e-27-Theta}, with transfer
function given by \eqref{e-28=W}.
Then its c-entropy is
\begin{equation}\label{entropy-Delta}
   \cS(\Theta_\Delta) = \frac{1}{2}\sum_{k=1}^m \ln\!\left(
          \frac{(\RE z_k)^2+(1+\IM z_k)^2}
               {(\RE z_k)^2+(1-\IM z_k)^2}
        \right).
\end{equation}
\end{theorem}
\begin{proof}
Write $z_k = x_k + i y_k$ with $x_k=\RE z_k\in\Bbb R$ and $y_k=\IM z_k>0$.
By Definition \ref{d-9-ent}, the c-entropy is
\[
   \cS(\Theta_\Delta) \;=\; -\ln\!\big|W_{\Theta_\Delta}(-i)\big|
   \qquad\text{(see \eqref{e-80-entropy-def}).}
\]
Using the explicit form of the transfer function of $\Theta_\Delta$ from \eqref{e-28=W}, that is,
\[
   W_{\Theta_\Delta}(z) \;=\; \prod_{k=1}^m \frac{z-\bar z_k}{z-z_k},
\]
we evaluate at $z=-i$:
\[
   \big|W_{\Theta_\Delta}(-i)\big|
   \;=\; \prod_{k=1}^m
          \left|\frac{-i-\bar z_k}{-i-z_k}\right|
   \;=\; \prod_{k=1}^m
          \left|\frac{-i-(x_k-iy_k)}{-i-(x_k+iy_k)}\right|
   \;=\; \prod_{k=1}^m
          \left|\frac{x_k-i(y_k-1)}{x_k+i(1+y_k)}\right|.
\]
Taking absolute values yields,
\[
   \left|\frac{x_k-i(y_k-1)}{x_k+i(1+y_k)}\right|
   \;=\; \sqrt{\frac{x_k^2+(y_k-1)^2}{x_k^2+(1+y_k)^2}}
   \;=\; \sqrt{\frac{(\RE z_k)^2+(1-\IM z_k)^2}{(\RE z_k)^2+(1+\IM z_k)^2}}.
\]
Hence
\[
   \big|W_{\Theta_\Delta}(-i)\big|
   \;=\; \prod_{k=1}^m
          \sqrt{\frac{(\RE z_k)^2+(1-\IM z_k)^2}{(\RE z_k)^2+(1+\IM z_k)^2}}.
\]
Using  properties of the logarithm we obtain
\[
    \begin{aligned}
   \cS(\Theta_\Delta)
   \;&=\; -\sum_{k=1}^m \tfrac12 \ln\!
          \left(\frac{(\RE z_k)^2+(1-\IM z_k)^2}{(\RE z_k)^2+(1+\IM z_k)^2}\right)\\
   \;&=\; \tfrac12\sum_{k=1}^m  \ln\!
          \left(\frac{(\RE z_k)^2+(1+\IM z_k)^2}{(\RE z_k)^2+(1-\IM z_k)^2}\right),
          \end{aligned}
\]
which is precisely \eqref{entropy-Delta}.
\end{proof}

\begin{remark}\label{r-7}
Note that each Blaschke factor
$$
  \frac{\bar z_k - z}{z_k - z}
$$
in \eqref{e-28=W} coincides with the transfer function of the one-dimensional multiplication L-system
\begin{equation}\label{e-32-Theta-k}
  \Theta_k =
   \begin{pmatrix}
      T_k & K_k & 1 \\
      \calH & & \dC
   \end{pmatrix}
  \end{equation}
of the form \eqref{mult-Theta} with parameter $\lambda_0=z_k$ (see the single-model formula \eqref{mult-W} in Section~\ref{s4}). By the Multiplication Theorem, (see Theorem \ref{unitar} in Appendix \ref{A}) coupling these systems $\Theta$ multiplies transfer functions as in \eqref{mult-theorem}:
\[
   W_{\Theta}(z)
   \;=\; \prod_{k=1}^m W_{\Theta_k}(z)= W_{\Theta_\Delta}(z)
   \]
Evaluating at $z=-i$ and using the definition of c-entropy together with the single-model entropy
formula \eqref{entropy-single},
\[
    \begin{aligned}
   \cS(\Theta_\Delta)   \;&=\; -\ln\!\big|W_{\Theta_\Delta}(-i)\big|
   \;=\; -\sum_{k=1}^m \ln\!\big|W_{\Theta_k}(-i)\big|
   \;=\; \sum_{k=1}^m \cS(\Theta_k)\\
   \;&=\; \tfrac12\sum_{k=1}^m  \ln\!\left(\frac{(\RE z_k)^2+(1+\IM z_k)^2}{(\RE z_k)^2+(1-\IM z_k)^2}
        \right).
        \end{aligned}
\]
This reproduces \eqref{entropy-Delta} and explicitly connects the result to the
factorization/coupling mechanism developed in Sections~\ref{s4}-\ref{s5}. Moreover, $\Theta_\Delta$ is minimal (see Section~\ref{s3}), and the coupled system $\Theta=\Theta_1\cdot\ldots\cdot\Theta_m$ is minimal by Lemma~\ref{lem:coupling-minimal}. Hence $\Theta_\Delta$ and $\Theta$ are unitarily equivalent.

\end{remark}
The correspondence between the interpolation system and its coupled multiplication
components established in Remark~\ref{r-7} naturally leads to a characterization of
how the location of the interpolation points affects the c-entropy.
The following theorem formulates these conditions explicitly.

\smallskip
\begin{theorem}
\label{t-8}
Let $\Theta_\Delta$ be the interpolation L-system in the model form \eqref{e-27-Theta}
corresponding to interpolation data $\{z_k\}_{k=1}^m \subset \dC_+$.
Then:
\begin{enumerate}
  \item[\textup{(a)}] The c-entropy $\cS(\Theta_\Delta)$ is \textbf{infinite} if and only if
  at least one node $z_k = i$.
\end{enumerate}
Assuming in addition that all nodes satisfy $z_k\neq i$, we have:
\begin{enumerate}
  \item[\textup{(b)}] The c-entropy $\cS(\Theta_\Delta)$ is \textbf{finite and maximal}
  precisely when all nodes $z_k$ are purely imaginary   $(\RE z_k = 0)$,

  \item[\textup{(c)}] If at least one node has $\RE z_k\neq 0$, then $\cS(\Theta_\Delta)$ is finite but strictly below its maximal value.
\end{enumerate}
\end{theorem}

\begin{proof}

Let $z_k=a_k+ib_k$ with $b_k>0$. For each $k$ consider the one-dimensional
multiplication-operator L-system $\Theta_k$ with parameter $\lambda_0=z_k$.
Its transfer function is $W_{\Theta_k}(z)=\frac{\bar z_k-z}{z_k-z}$ (see \eqref{mult-W}),
which coincides with the $k$-th Blaschke factor of $W_{\Theta_\Delta}(z)$ in
\eqref{e-28=W}. By the Multiplication Theorem, coupling the systems $\{\Theta_k\}$
produces a canonical L-system $\Theta$ whose transfer function satisfies
\[
W_\Theta(z)=\prod_{k=1}^m W_{\Theta_k}(z)=W_{\Theta_\Delta}(z).
\]
In particular,
\begin{equation}\label{e-33}
\cS(\Theta_\Delta)=-\ln|W_{\Theta_\Delta}(-i)|
=-\sum_{k=1}^m \ln|W_{\Theta_k}(-i)|
=\sum_{k=1}^m \cS(\Theta_k).
\end{equation}
Remark~9 in~\cite{BT-23} asserts that $\cS(\Theta_k)$ is infinite exactly when $z_k=i$,
and that for fixed $b_k>0$ the finite maximum is attained when $\RE z_k=0$.
Part \textup{(a)} follows from \eqref{e-33}: the sum is infinite iff at least one summand $\cS(\Theta_k)$ is infinite, which occurs exactly when  $z_k=i$.

Under the standing assumption $z_k\neq i$ for all $k$, each $\cS(\Theta_k)$ is finite. For fixed $b_k=\IM z_k$, the single-node entropy $\cS(\Theta_k)$ is maximized at $\RE z_k=0$ and strictly decreases when $\RE z_k\neq 0$ by Remark~9 in~\cite{BT-23}. Summing \eqref{e-33} over $k$ yields \textup{(b)} and \textup{(c)}.
\end{proof}

\begin{remark}\label{r-9}
The conditions of Theorem~\ref{t-8} admit a clear geometric interpretation
in the upper half-plane. Each interpolation point $z_k \in \dC_+$ determines a corresponding
parameter $\lambda_k =z_k= a_k + i b_k$ of the multiplication model, where $b_k = \IM \lambda_k > 0$
quantifies the local dissipation. The value of the c-entropy reflects how far the point $z_k$
lies from the imaginary axis:

\begin{itemize}
    \item If $\RE \lambda_k = 0$, then $z_k$ lies on the vertical line through the origin,
    and the corresponding L-system contributes its \emph{maximal finite} share of entropy.
    \item If $\RE \lambda_k \neq 0$, the point $z_k$ departs from this axis,
    the dissipative part decreases, and the contribution to $\cS$ diminishes accordingly.
    \item When $\lambda_k = i$, the interpolation node reaches the extreme point
    of maximal dissipation in $\dC_+$, producing \emph{infinite c-entropy}.
\end{itemize}

Hence, the global c-entropy of $\Theta_\Delta$ is governed by the collective geometry of the
interpolation nodes $\{z_k\}$ in $\dC_+$, with the imaginary parts controlling the magnitude
of dissipation and the real parts determining how close the system is to its maximal or
infinite-entropy regimes.
\end{remark}

\smallskip

The geometric insight of Remark~\ref{r-9} clarifies how the analytic
placement of the interpolation nodes in $\dC_+$ controls both dissipation and entropy.
In the next section, we construct the explicit interpolation model $\Theta_\Delta$
and examine how these geometric and functional invariants are realized in its
operator-theoretic structure.

\medskip
The dissipation coefficient follows immediately from the definition $\cD=1-e^{-2\cS}$.

\begin{theorem}\label{main-thm-D}
For the interpolation L-system $\Theta_\Delta$ the dissipation coefficient is
\begin{equation}\label{diss-Delta}
   \cD(\Theta_\Delta) =
   1 - \prod_{k=1}^m
        \frac{(\RE z_k)^2+(1-\IM z_k)^2}
             {(\RE z_k)^2+(1+\IM z_k)^2}.
\end{equation}
\end{theorem}
\begin{proof}
Recall the definition \ref{d-10} of  $\cD=1-e^{-2\cS}$ (see Section~\ref{s2}) and the entropy formula
\eqref{entropy-Delta} established above. Then
\[
   e^{-2\cS(\Theta_\Delta)}
   \;=\; \prod_{k=1}^m
          \frac{(\RE z_k)^2+(1-\IM z_k)^2}{(\RE z_k)^2+(1+\IM z_k)^2},
\]
and hence
\[
   \cD(\Theta_\Delta)
   \;=\; 1 - e^{-2\cS(\Theta_\Delta)}
   \;=\; 1 - \prod_{k=1}^m
          \frac{(\RE z_k)^2+(1-\IM z_k)^2}{(\RE z_k)^2+(1+\IM z_k)^2},
\]
which is precisely \eqref{diss-Delta}.
\end{proof}

We can also note that by Section \ref{s3} we have the Blaschke product representation
\(
   W_{\Theta_\Delta}(z)
   = \prod_{k=1}^m \frac{z-\bar z_k}{z-z_k}
\)
(see \eqref{e-28=W}), realized as a coupling of one-dimensional multiplication L-systems
with parameters $\lambda_0=z_k$ (Section~\ref{s4}, \eqref{mult-W}, \eqref{mult-theorem}). Evaluating c-entropy
for the single model (Section~\ref{s5}) gives
\[
   \cS_k \;=\; \frac{1}{2}\ln\!\left(
          \frac{(\RE z_k)^2+(1+\IM z_k)^2}
               {(\RE z_k)^2+(1-\IM z_k)^2}
        \right)
   \quad\text{and}\quad
   \cD_k \;=\; \frac{4\,\IM z_k}{(\RE z_k)^2+(1+\IM z_k)^2},
\]
cf.\ \eqref{entropy-single}, \eqref{diss-single}. The coupling law for dissipation
(Section~\ref{s5}) states
\[
   \cD \;=\; 1 - \prod_{k=1}^m (1-\cD_k),
   \qquad\text{see }\eqref{diss-coupling-general}.
\]
But for each $k$,
\[
   1-\cD_k
   \;=\; 1 - \frac{4\,\IM z_k}{(\RE z_k)^2+(1+\IM z_k)^2}
   \;=\; \frac{(\RE z_k)^2+(1-\IM z_k)^2}{(\RE z_k)^2+(1+\IM z_k)^2}.
\]
Therefore
\[
   \cD(\Theta_\Delta)
   \;=\; 1 - \prod_{k=1}^m
          \frac{(\RE z_k)^2+(1-\IM z_k)^2}{(\RE z_k)^2+(1+\IM z_k)^2},
\]
which coincides with \eqref{diss-Delta}. \qedhere

\noindent
These formulas express the invariants of the interpolation solution entirely in terms of the
interpolation nodes $\{z_k\}$, confirming the additivity principle of c-entropy under
coupling and its consistency with the multiplicative structure of transfer functions.

It is instructive to examine the special case when all interpolation nodes lie on the imaginary axis.
This situation not only simplifies the general formulas but also illustrates clearly how individual nodes
contribute to the overall entropy and dissipation. In particular, it highlights the additive behavior for
small imaginary parts and the scaling effect when multiple identical nodes are present.
\begin{corollary}\label{c-7}
Assume $z_k = i a_k$ with $a_k>0$ for $k=1,\dots,m$. Then the c-entropy and dissipation coefficient of the interpolation system $\Theta_\Delta$ are
\begin{equation}\label{entropy-Delta-imag}
   \cS(\Theta_\Delta) \;=\; \sum_{k=1}^m \ln\!\left(\frac{1+a_k}{\,1-a_k\,}\right),
\end{equation}
\begin{equation}\label{diss-Delta-imag}
   \cD(\Theta_\Delta) \;=\; 1 \;-\; \prod_{k=1}^m \left(\frac{1-a_k}{\,1+a_k\,}\right)^{\!2}.
\end{equation}
\end{corollary}

\noindent\textit{Proof.}
Substitute $\RE z_k=0$ and $\IM z_k=a_k$ in \eqref{entropy-Delta} and \eqref{diss-Delta}.
Since $(\RE z_k)^2+(1\pm\IM z_k)^2=(1\pm a_k)^2$, we get
\[
\cS(\Theta_\Delta)
= \sum_{k=1}^m \tfrac{1}{2}\ln\!\left(\frac{(1+a_k)^2}{(1-a_k)^2}\right)
= \sum_{k=1}^m \ln\!\left(\frac{1+a_k}{1-a_k}\right),
\]
and then $\cD(\Theta_\Delta)=1-e^{-2\cS(\Theta_\Delta)}$ yields \eqref{diss-Delta-imag}.
\hfill$\Box$

\medskip
\noindent\textbf{Remarks.}
\begin{itemize}
   \item[(i)] When all $a_k$ coincide, say $a_k=a$, the formulas \eqref{entropy-Delta-imag} and \eqref{diss-Delta-imag} reduce to
   \[
      \cS(\Theta_\Delta) = m\,\ln\!\left(\frac{1+a}{1-a}\right), \qquad
      \cD(\Theta_\Delta) = 1 - \left(\frac{1-a}{1+a}\right)^{2m}.
   \]
   This shows that coupling $m$ identical imaginary nodes simply scales the entropy linearly in $m$, while the dissipation coefficient accumulates according to the nonlinear law of Section~\ref{s5}.

   \item[(ii)] For small values $a_k\ll 1$, we may expand around $a_k=0$ to get
   \[
      \ln\!\left(\frac{1+a_k}{1-a_k}\right) = 2a_k + \tfrac{2}{3}a_k^3 + O(a_k^5),
   \]
   so that
   \[
      \cS(\Theta_\Delta) = 2\sum_{k=1}^m a_k + O\!\Big(\sum_{k=1}^m a_k^3\Big),
      \qquad
      \cD(\Theta_\Delta) = 4\sum_{k=1}^m a_k + O\!\Big(\sum_{k=1}^m a_k^2\Big).
   \]

   \smallskip
\noindent
Consequently, the leading-order linear dependence of $S(\Theta_\Delta)$ and $D(\Theta_\Delta)$ on the imaginary parts $a_k$ reflects the weak-dissipation regime of the corresponding multiplication-operator components.
   Thus, in the case of weakly dissipative nodes (small imaginary parts), both invariants behave additively to leading order. This reflects the fact that each interpolation node contributes approximately independently to the overall entropy and dissipation.
\end{itemize}

We next give a precise characterization of the interpolation data corresponding to the minimal regime of the c-entropy.

\medskip

\begin{proposition}\label{p-12}
Let $\{z_k\}_{k=1}^m \subset \dC_+$ be interpolation nodes with $z_k=a_k+ib_k$ and $b_k>0$,
and let $\Theta_\Delta$ be the interpolation L-system of the form \eqref{e-27-Theta} corresponding to
the data $v_k=i$. Then
\begin{equation}\label{eq:S-sum}
\cS(\Theta_\Delta)
= -\ln|W_{\Theta_\Delta}(-i)|
= \frac12\sum_{k=1}^m  \ln\!\left(
\frac{a_k^2+(1+b_k)^2}{a_k^2+(1-b_k)^2}
\right)\ \ge\ 0.
\end{equation}
Moreover:
\begin{enumerate}
\item[\textup{(i)}] The \emph{infimum} of $\cS(\Theta_\Delta)$ over admissible nodes is $0$,
and it is not attained for any $b_k>0$.
\item[\textup{(ii)}] For each $k$,
$\displaystyle  \ln\!\left(\frac{a_k^2+(1+b_k)^2}{a_k^2+(1-b_k)^2}\right)\downarrow 0$
as either $b_k\downarrow 0$ or $|a_k|\to\infty$. Consequently,
$\cS(\Theta_\Delta)\downarrow 0$ precisely when every node approaches the real axis
($b_k\downarrow 0$) or escapes horizontally ($|a_k|\to\infty$).
\end{enumerate}
\end{proposition}

\begin{proof}
For the model L-system  $\Theta_\Delta$ of the form \eqref{e-27-Theta} described in Section~\ref{s3} (with target values $v_k=i$), the transfer function is given by \eqref{e-28=W}, that is
\[
W_{\Theta_\Delta}(z)=\prod_{k=1}^m\frac{z-\bar z_k}{z-z_k}.
\]
Evaluating at $z=-i$ gives
\[
\left|\frac{-i-\bar z_k}{-i-z_k}\right|
=\left|\frac{a_k-i(1+b_k)}{a_k-i(1-b_k)}\right|
=\sqrt{\frac{a_k^2+(1+b_k)^2}{a_k^2+(1-b_k)^2}}.
\]
Hence
\[
-\ln|W_{\Theta_\Delta}(-i)|
=\sum_{k=1}^m -\ln\left|\frac{-i-\bar z_k}{-i-z_k}\right|
=\sum_{k=1}^m \frac12 \ln\!\left(
\frac{a_k^2+(1+b_k)^2}{a_k^2+(1-b_k)^2}
\right),
\]
confirming \eqref{eq:S-sum}. Since $(1+b_k)^2\ge (1-b_k)^2$ for $b_k>0$, every summand is
nonnegative, so $\cS(\Theta_\Delta)\ge 0$.

If $\cS(\Theta_\Delta)=0$, then each summand must vanish, i.e., $$a_k^2+(1+b_k)^2=a_k^2+(1-b_k)^2,$$ which forces $b_k=0$ for all $k$ but this contradicts $z_k\in\dC_+$ ($b_k>0$). Therefore the infimum over admissible nodes is $0$, but it is not
attained (claim (i)).

In order to prove (ii), fix $a_k$ and let $b_k\downarrow 0$. Then
\[
\frac{a_k^2+(1+b_k)^2}{a_k^2+(1-b_k)^2}\ \longrightarrow\ 1,
\quad\text{hence}\quad
\frac12\ln\!\left(\frac{a_k^2+(1+b_k)^2}{a_k^2+(1-b_k)^2}\right)\ \longrightarrow\ 0.
\]
Alternatively, for fixed $b_k>0$, as $|a_k|\to\infty$ one has
\[
\frac{a_k^2+(1+b_k)^2}{a_k^2+(1-b_k)^2}
=1+\frac{(1+b_k)^2-(1-b_k)^2}{a_k^2+(1-b_k)^2}
\ \longrightarrow\ 1,
\]
and the same logarithmic term tends to $0$. Summing over $k$ yields the stated characterization $\cS(\Theta_\Delta)\downarrow 0$.

Thus, the limit $\cS(\Theta_\Delta)\to 0$ characterizes the \emph{minimal regime} of the interpolation system, realized when the interpolation nodes approach the real axis ($b_k\downarrow 0$) or drift horizontally to infinity ($|a_k|\to\infty$).
\end{proof}

\begin{remark}\label{r-13}
In the special case $v_\ell=i$, the c-entropy admits the factorized form
\[
\cS(\Theta_\Delta)
=\sum_{k=1}^m \ln\left|\frac{z_k+i}{\,z_k-i\,}\right|.
\]
Hence $\cS(\Theta_\Delta)\downarrow 0$ precisely when each node $z_k$ approaches the real axis
($\IM z_k\downarrow 0$) or escapes horizontally ($|\,\RE z_k\,|\to\infty$). Geometrically,
this is the regime where the Blaschke factors $\frac{z-\bar z_k}{z-z_k}$ have modulus
$\big|\frac{-i-\bar z_k}{-i-z_k}\big|\to 1$ at $z=-i$, so that $|W(-i)|\to 1$ and
$-\ln|W(-i)|\to 0$. By contrast, placing a node at $z_k=i$ forces $\cS(\Theta_\Delta)=\infty$, while aligning nodes on the imaginary axis ($\RE z_k=0$ with $z_k\neq i$) yields the \emph{maximal finite} regime discussed earlier.
\end{remark}
These asymptotic observations complete the characterization of all entropy regimes (minimal, maximal, and infinite) and prepare the ground for the geometric interpretation developed in Section~\ref{s8}.

\medskip
The explicit formulas derived above for the c-entropy and dissipation coefficient of the
interpolation system $\Theta_\Delta$, together with the geometric characterization in
Theorem~\ref{t-8} and Remark~\ref{r-9}, reveal how the location of interpolation nodes in
$\dC_+$ governs the transition between finite, maximal, and infinite c-entropy regimes.  In
particular, purely imaginary nodes yield the maximal finite c-entropy, while any horizontal
displacement reduces it, and the critical node $z_k = i$ produces divergence of c-entropy.
These results demonstrate that the invariants $\mathcal{S}$ and $\mathcal{D}$ capture
intrinsic geometric information about the interpolation data, linking analytic placement of
nodes to the dissipative structure of the corresponding L-system.



\bigskip
\noindent

The next section focuses on the special case of symmetric interpolation configurations,
which represent the equilibrium regime of this balance and naturally admit
a physical interpretation in terms of energy exchange within corresponding $LC$-network
analogs.

\section{Symmetric Interpolation, Physical Interpretation, Examples}\label{s7}

We now examine \emph{symmetric interpolation configurations}, in which the interpolation nodes are placed symmetrically with respect to the imaginary axis.
This symmetry plays a distinguished role: it yields explicit closed forms for both the transfer and impedance functions of the corresponding interpolation L-system~$\Theta_\Delta$, and it links the analytic geometry of interpolation data with the energy balance of the associated physical network.

We begin with the simplest nontrivial configuration, a single symmetric pair of nodes.
In this case, the transfer function $W_{\Theta_\Delta}(z)$ and the impedance function $V_{\Theta_\Delta}(z)$ are those of the \emph{canonical interpolation L-system} $\Theta_\Delta$ of the form \eqref{e-27-Theta}, related by the Cayley transform \eqref{e-8-WV}.
The following lemma provides their explicit analytic form and serves as a prototype for the general multi-pair constructions developed later in Theorem~\ref{t:explicitABa}.

\begin{lemma}[Single symmetric pair of interpolation nodes]\label{l:symm-pair-cayley}
Let the interpolation data consist of a symmetric pair of nodes
\[
   z_1 = a + i b, \qquad z_2 = -a + i b,
   \qquad a \ne 0,\; b > 0.
\]
Consider the transfer function $W_{\Theta_\Delta}(z)$ and the impedance function
$V_{\Theta_\Delta}(z)$ of the interpolation L-system~$\Theta_\Delta$  of the form \eqref{e-27-Theta}.
Then
\begin{equation}\label{e:pair-form}
   W_{\Theta_\Delta}(z)
   = \frac{(z-\bar z_1)(z+z_1)}{(z-z_1)(z+\bar z_1)},\qquad
   V_{\Theta_\Delta}(z)
   = \frac{Az}{B^2 - z^2},
\end{equation}
where
\[
   A = 2b,\qquad B = \sqrt{a^2 + b^2}.
\]
In particular, $V_{\Theta_\Delta}(z_k)=i$ for $k=1,2$, and $V_{\Theta_\Delta}(z)$ is a rational Herglotz-Nevanlinna function.
\end{lemma}

\begin{proof}
Using the formula \eqref{e-28=W} of the transfer function $W_{\Theta_\Delta}(z)$ for the symmetric pair $(z_1,-\bar z_1)$, we get
\[
   W_{\Theta_\Delta}(z)
   = \frac{(z-\bar z_1)(z+z_1)}{(z-z_1)(z+\bar z_1)}
   = \frac{(z-a+ib)(z+a+ib)}{(z-a-ib)(z+a-ib)}.
\]
Substituting $W_{\Theta_\Delta}(z)$ above into the Cayley transform \eqref{e-8-WV} yields
\[
   V_{\Theta_\Delta}(z)
   = i\,\frac{W_{\Theta_\Delta}(z)-1}{W_{\Theta_\Delta}(z)+1}
   = \frac{2b\,z}{a^2 + b^2 - z^2}
   = \frac{Az}{B^2 - z^2},
\]
with $A=2b$ and $B^2=a^2+b^2$.
To verify the Herglotz-Nevanlinna property, write $z=x+iy$ with $y>0$, and set
$B^2 = a^2+b^2$ and $P = B^2 - x^2 + y^2$. Then
\[
   V_{\Theta_\Delta}(z)
   = A\,\frac{x+iy}{P - 2i x y}
   = A\,\frac{(x+iy)(P+2i x y)}{P^2 + 4x^2y^2}.
\]
A short computation shows
\[
   (x+iy)(P+2i x y)
   = \underbrace{x\bigl(B^2 - x^2 - y^2\bigr)}_{\text{real}}
     \;+\; i\,\underbrace{y\bigl(B^2 + x^2 + y^2\bigr)}_{\text{imag}},
\]
hence
\begin{equation}\label{e:ImV-positive}
   \IM V_{\Theta_\Delta}(x+iy)
   = \frac{2by\,(B^2 + x^2 + y^2)}{(B^2 - x^2 + y^2)^2 + 4x^2y^2}.
   \end{equation}
Since $b>0$, $y>0$, and the denominator is strictly positive, \eqref{e:ImV-positive} is strictly positive.
Therefore $\IM V_{\Theta_\Delta}(z)>0$ for $\IM z>0$, and $V_{\Theta_\Delta}$ is Herglotz-Nevanlinna.
Finally, evaluating at $z=z_1$ and $z=z_2$ gives $V_{\Theta_\Delta}(z_k)=i$, $k=1,2$.
\end{proof}
\smallskip
The explicit computation in Lemma~\ref{l:symm-pair-cayley} shows that a symmetric pair of interpolation nodes produces a rational Herglotz-Nevanlinna impedance function of the form $A\,z/(B^2-z^2)$, which realizes the canonical interpolation L-system $\Theta_\Delta$. Adding further nodes perturbs this structure in a controlled
way, as illustrated next.
\begin{corollary}[One symmetric pair plus one purely imaginary node]\label{c:pair-plus-imag}
Let the interpolation data consist of three nodes
\[
   z_1=a+ib,\qquad z_2=-a+ib,\qquad z_3=ic,
   \qquad a\neq 0,\; b>0,\; c>0.
\]
Let $W_{\Theta_\Delta}(z)$ be the transfer function of the canonical interpolation L-system $\Theta_\Delta$ of the form \eqref{e-27-Theta} for this data and $V_{\Theta_\Delta}(z)$ its impedance function.
Then $V_{\Theta_\Delta}(z)$ has the two-term rational form
\begin{equation}\label{e:two-term-form}
   V_{\Theta_\Delta}(z)
   = -\frac{C}{z}\;+\; \frac{Az}{\widetilde B^{\,2}-z^2},
\end{equation}
where, with $B^2:=a^2+b^2$,
\begin{equation}\label{e:two-term-params}
   \widetilde B^{\,2} \;=\; B^2 + 2bc,
   \qquad
   C \;=\; \frac{c\,B^2}{\widetilde B^{\,2}},
   \qquad
   A \;=\; (2b+c) - C \;=\; (2b+c) - \frac{c\,B^2}{\widetilde B^{\,2}}.
\end{equation}
In particular,
\[
   V_{\Theta_\Delta}(z_k)=i,\qquad k=1,2,3,
\]
and $V_{\Theta_\Delta}$ is Herglotz-Nevanlinna on $\dC_+$.
\end{corollary}

\begin{proof}
The transfer function is the product of the single-pair factor and the purely imaginary node factor:
\[
   W_{\Theta_\Delta}(z)
   = \underbrace{\frac{(z-\bar z_1)(z+z_1)}{(z-z_1)(z+\bar z_1)}}_{\text{symmetric pair}}
     \cdot
     \underbrace{\frac{z-\bar z_3}{\,z-z_3\,}}_{\text{imaginary node}}
   = \frac{(z-a+ib)(z+a+ib)}{(z-a-ib)(z+a-ib)}\cdot\frac{z+ic}{z-ic}.
\]
Applying the Cayley transform \eqref{e-8-WV} and simplifying gives
\begin{equation}\label{e:pre-PFD}
   V_{\Theta_\Delta}(z)
   = i\,\frac{W_{\Theta_\Delta}(z)-1}{W_{\Theta_\Delta}(z)+1}
   = \frac{a^2c+b^2c-(2b+c)z^2}{\,z\bigl(z^2-(a^2+b^2+2bc)\bigr)}.
\end{equation}
Set $B^2:=a^2+b^2$ and $\widetilde B^{\,2}:=B^2+2bc$; then \eqref{e:pre-PFD} becomes
\[
   V_{\Theta_\Delta}(z)=\frac{cB^2-(2b+c)z^2}{\,z\bigl(z^2-\widetilde B^{\,2}\bigr)}.
\]

We seek a decomposition of the form \eqref{e:two-term-form}.
Bringing the right-hand side over the common denominator $z(\widetilde B^{\,2}-z^2)$ gives
\[
   -\frac{C}{z} + A\,\frac{z}{\widetilde B^{\,2}-z^2}
   = \frac{-C(\widetilde B^{\,2}-z^2)+A z^2}{\,z(\widetilde B^{\,2}-z^2)}
   = -\,\frac{cB^2-(2b+c)z^2}{\,z(\widetilde B^{\,2}-z^2)}.
\]
Equating numerators yields
\[
   -\bigl(cB^2-(2b+c)z^2\bigr)
   = -C\widetilde B^{\,2} + (C+A)z^2.
\]
Matching constant and $z^2$ coefficients gives
\[
   C\,\widetilde B^{\,2}=cB^2 \ \Longrightarrow\
   C=\frac{cB^2}{\widetilde B^{\,2}},
   \qquad
   C+A=2b+c \ \Longrightarrow\
   A=(2b+c)-C,
\]
which proves \eqref{e:two-term-form}-\eqref{e:two-term-params}.

For the Herglotz-Nevanlinna property, note $0<C<c$ since $\widetilde B^{\,2}>B^2$,
and thus $A=(2b+c)-C>2b>0$. For $z=x+iy$ with $y>0$,
\[
   \Im\!\left(-\frac{C}{z}\right) = \frac{Cy}{x^2+y^2}>0,
   \quad
   \Im\!\left(A\,\frac{z}{\widetilde B^{\,2}-z^2}\right)
   = A\,\frac{y\bigl(\widetilde B^{\,2}+x^2+y^2\bigr)}{\bigl(\widetilde B^{\,2}-x^2+y^2\bigr)^2+4x^2y^2}>0,
\]
hence $\Im V_{\Theta_\Delta}(z)>0$ on $\dC_+$. Evaluating at $z=z_1,z_2,ic$ gives $V_{\Theta_\Delta}(z_k)=i$.
\end{proof}

\smallskip
The preceding results, Lemma~\ref{l:symm-pair-cayley} and
Corollary~\ref{c:pair-plus-imag}, reveal a consistent analytic structure.
Each symmetric pair of interpolation nodes contributes a rational
Herglotz--Nevanlinna component of the form $A_k\,z/(B_k^2-z^2)$,
while purely imaginary nodes introduce additional terms of type
$-A_0/z$.
When several symmetric pairs are present, their individual contributions combine additively in the model construction by additive superposition of the pair contributions. These observations motivate the introduction of a single analytic function that captures the overall resonant-capacitive composition of the data. This function, denoted by $V_\Delta(z)$, does not necessarily interpolate the prescribed values $i$ at all nodes for $n>1$, but it reproduces the essential analytic and geometric structure of the corresponding interpolation system and provides a convenient model for the study of symmetric configurations.

\begin{definition}[Model impedance function]\label{d:model-impedance}
Let the interpolation data consist of $n$ symmetric pairs
$(a_k \pm ib_k)$ with $a_k \ne 0$, $b_k > 0$, and $m$ purely imaginary
nodes $i b_j$ with $b_j > 0$.
The associated \emph{model impedance function} is defined by
\begin{equation}\label{e:Vgeneral}
   V_\Delta(z)
   = -\frac{A_0}{z}
     + \sum_{k=1}^{n} A_k\,\frac{z}{B_k^2 - z^2},
   \quad
   A_0 = \sum_{j=1}^{m} b_j,\;
   A_k = 2b_k,\;
   B_k = \sqrt{a_k^2 + b_k^2}.
\end{equation}
Here, the first term represents the collective capacitive contribution of all
purely imaginary nodes, while each summand in the series corresponds to one
symmetric resonant branch of the interpolation L-system.
\end{definition}

\smallskip
The function $V_\Delta(z)$ introduced in Definition~\ref{d:model-impedance} provides an analytic model
for symmetric interpolation configurations. Although it does not, in general, realize the interpolation conditions $V_\Delta(z_k)=i$ for all nodes, it retains the essential analytic structure determined by the symmetric placement of the data. In particular, each pair $(z_k,-\bar z_k)$ contributes a term of the form
$A_k z/(B_k^2-z^2)$, and the total model impedance is obtained by additive coupling of these components.
The following theorem gives the explicit formula for $V_\Delta(z)$ and clarifies its relation to the transfer function of the canonical interpolation L-system~$\Theta_\Delta$.

\begin{theorem}[Symmetric Pairs of Interpolation Nodes]\label{t:explicitABa}
Let the interpolation data consist of $2n$ symmetric nodes in $\dC_+$ satisfying
\[
   z_k = a_k + i b_k,\qquad
   z_{2n+1-k} = -a_k + i b_k,\qquad
   a_k\in\dR,\; b_k>0,\; a_k\ne0.
\]
Then the \emph{model impedance function} $V_\Delta(z)$ is given by
\begin{equation}\label{e:VexplicitA}
   V_\Delta(z)=\sum_{k=1}^{n}A_k\,\frac{z}{B_k^2-z^2},
\end{equation}
where
\begin{equation}\label{e:ABformulaA}
   A_k=2b_k,\qquad
   B_k=\sqrt{a_k^2+b_k^2}>0.
\end{equation}
Each term in \eqref{e:VexplicitA} is a rational Herglotz--Nevanlinna component associated with one symmetric pair of interpolation nodes.
\end{theorem}

\begin{proof}
For each symmetric pair $(z_k,-\bar z_k)$ with $z_k=a_k+ib_k$ and $b_k>0$, the transfer factor is (see Lemma \ref{l:symm-pair-cayley} and \eqref{e:pair-form})
\[
   W_k(z)=\frac{(z-\bar z_k)(z+z_k)}{(z-z_k)(z+\bar z_k)}.
\]
Applying the Cayley transform \eqref{e-8-WV} to $W_k$ gives the
Herglotz--Nevanlinna component
\[
   V_k(z)= i\,\frac{W_k(z)-1}{W_k(z)+1}
          = \frac{2b_k\,z}{a_k^2+b_k^2-z^2}
          = A_k\,\frac{z}{B_k^2-z^2},
\]
with $A_k=2b_k$ and $B_k=\sqrt{a_k^2+b_k^2}$.
By Definition~\ref{d:model-impedance} (with no purely imaginary nodes),
the model impedance for $n$ symmetric pairs is the sum of these components:
\[
   V_\Delta(z)=\sum_{k=1}^{n} V_k(z)
              =\sum_{k=1}^{n} A_k\,\frac{z}{B_k^2-z^2},
\]
which is precisely the formula \eqref{e:VexplicitA}. Each $V_k$ belongs to the Herglotz--Nevanlinna class and satisfies $V_k(z_k)=i$ at its interpolation nodes. Hence $V_\Delta(z)$, being their sum, also belongs to the
Herglotz--Nevanlinna class.
\end{proof}
The corresponding transfer function of the canonical interpolation L-system~$\Theta_\Delta$ is
\begin{equation}\label{e-49-sym-W}
   W_{\Theta_\Delta}(z)
   =\prod_{k=1}^{n}
     \frac{(z-\bar z_k)(z+z_k)}{(z-z_k)(z+\bar z_k)}.
\end{equation}
This product representation contrasts with the additive structure of the model impedance function~\eqref{e:VexplicitA}.
The following remark clarifies the analytic relationship between the two
and explains how the model form $V_\Delta(z)$ relates to the canonical
pair $(V_{\Theta_\Delta}(z),W_{\Theta_\Delta}(z))$.

\begin{remark}[Cayley structure of the model function]\label{r:Cayley-structure}
The impedance function $V_{\Theta_\Delta}(z)$ and the transfer function $W_{\Theta_\Delta}(z)$ of the canonical interpolation L-system are linked by the Cayley transform~\eqref{e-8-WV}. In contrast, the model impedance function $V_\Delta(z)$ of Theorem~\ref{t:explicitABa} possesses an additive representation that is not obtained by applying this transform directly to the product form \eqref{e-49-sym-W}. Instead, $V_\Delta(z)$ serves as a simplified analytic realization built by additive coupling of the individual symmetric-pair components $V_k(z)=\dfrac{2b_k\,z}{B_k^2-z^2}$. It reproduces the global resonant-capacitive structure of the interpolation data, although for $n>1$ it does not, in general, realize the exact Nevanlinna--Pick values $V_\Delta(z_k)=i$ satisfied by $V_{\Theta_\Delta}(z)$.
\end{remark}

\smallskip
The next step is to extend this construction to configurations that include purely imaginary interpolation nodes not paired by symmetry. Such nodes represent the limiting case of vanishing real parts in the symmetric setting and contribute additional terms of capacitive type to the model impedance function. The following theorem provides the explicit analytic form of $V_\Delta(z)$ in this degenerate situation.

\begin{theorem}[Purely Imaginary Interpolation Nodes]\label{t:explicitABb}
Suppose the interpolation data contain $m$ purely imaginary nodes
\[
   z_j = i\,b_j,\qquad b_j>0,
\]
without symmetric counterparts. Then the associated model impedance function $V_0(z)$, which is Herglotz--Nevanlinna and satisfies $V_0(z_j)=i$, has the form
\begin{equation}\label{e:V0}
   V_0(z)
   = -\sum_{j=1}^{m}\frac{A_{0j}}{z},\qquad
     A_{0j}=b_j.
\end{equation}
Each term in~\eqref{e:V0} represents the degenerate limit of a symmetric pair
with vanishing real part $a_k\to0$,
a term of purely capacitive type in the sense of the subsequent
physical interpretation (see Remark~\ref{r:LC}).
\end{theorem}

\begin{proof}
For a single purely imaginary node $z_j = i\,b_j$ with $b_j>0$, the corresponding
transfer function of the canonical interpolation L-system is
\[
   W_j(z)
   = \frac{z - \bar z_j}{z - z_j}
   = \frac{z + i b_j}{z - i b_j}.
\]
Applying the Cayley transform~\eqref{e-8-WV} gives
\[
   V_j(z)
   = i\,\frac{W_j(z)-1}{W_j(z)+1}
   = -\,\frac{b_j}{z}.
\]
Summing over all $m$ purely imaginary nodes yields
\[
   V_0(z)
   = -\sum_{j=1}^{m}\frac{b_j}{z},
\]
which coincides with~\eqref{e:V0}.
Each term $-b_j/z$ represents the degenerate limit of the symmetric-pair
expression $A_k\,z/(B_k^2-z^2)$ as the real part $a_k\to0$,
a term of purely capacitive type in the sense of the subsequent
physical interpretation (see Remark~\ref{r:LC}).
\end{proof}

\smallskip
Combining the results of Theorem~\ref{t:explicitABa} and
Theorem~\ref{t:explicitABb} yields the general analytic form of the
model impedance function for symmetric interpolation configurations
that may include both symmetric pairs and purely imaginary nodes.
The resulting expression unifies the resonant and capacitive components
into a single rational Herglotz--Nevanlinna function, as stated below.

\begin{corollary}[General Symmetric Configuration]\label{c:Vgeneral}
Let the interpolation data consist of $n$ symmetric pairs
$(a_k\pm i b_k)$ with $a_k\ne0$, $b_k>0$, and $m$ purely imaginary
nodes $i\,b_j$ with $b_j>0$.
Then the combined model impedance function is
\begin{equation}\label{e:Vgeneral-final}
   V_\Delta(z)
   = -\frac{A_0}{z}
     + \sum_{k=1}^{n} A_k\,\frac{z}{B_k^2 - z^2},
\end{equation}
where
\[
   A_0 = \sum_{j=1}^{m} b_j,\qquad
   A_k = 2b_k,\qquad
   B_k = \sqrt{a_k^2+b_k^2}>0.
\]
The first term in~\eqref{e:Vgeneral-final} represents the collective
contribution of all purely imaginary nodes, while each summand in the
series corresponds to one symmetric resonant branch of the model
impedance function.
Together these terms form a rational Herglotz--Nevanlinna function that
encodes the resonant--capacitive structure of the interpolation data.
(The canonical impedance $V_{\Theta_\Delta}(z)$ obtained via
\eqref{e-8-WV} satisfies $V_{\Theta_\Delta}(z_k)=i$ for all nodes.)
\end{corollary}

\begin{proof}
The expression~\eqref{e:Vgeneral-final} follows directly by superposition of the symmetric-pair representation established in Theorem~\ref{t:explicitABa} and the purely imaginary contribution described in Theorem~\ref{t:explicitABb}.
\end{proof}

\smallskip
Theorems~\ref{t:explicitABa}--\ref{t:explicitABb}, together with Corollary~\ref{c:Vgeneral}, provide a complete analytic description of symmetric interpolation systems, encompassing both resonant (sym\-metric-pair) and purely reactive (imaginary-node) configurations. Beyond their analytic meaning, these results also admit a direct
\emph{physical interpretation} in terms of classical network synthesis, where each term in~\eqref{e:Vgeneral-final} corresponds to a distinct component of a passive $LC$ circuit. The following remark makes this correspondence explicit and explains how the interpolation data determine the energy storage and exchange properties of the associated system.


\begin{remark}[Physical interpretation]\label{r:LC}
The \emph{model impedance function} $V_\Delta(z)$ introduced in
Corollary~\ref{c:Vgeneral},
\[
   V_\Delta(z)
   = -\frac{A_0}{z}
     + \sum_{k=1}^n A_k\,\frac{z}{B_k^2 - z^2},
\]
admits a direct interpretation within the framework of classical network
synthesis, where the analytic terms correspond to the capacitive and
resonant elements of a passive electrical circuit.
To make this correspondence explicit, it is convenient to rewrite
$V_\Delta(z)$ in the real-frequency variable $p=-iz$,
so that
\[
   Z(p)
   = \frac{1}{i}\,V_\Delta(i p)
   = \frac{A_0}{p}
     + \sum_{k=1}^n A_k\,\frac{p}{B_k^2 + p^2},
\]
a rational positive-real function analytic in the right half-plane
$\Re p>0$ and real on the positive real axis
(see Remark~1 of~\cite{BT-23}).
Consequently, $Z(p)$ represents the impedance of a passive one-port
electrical network consisting of a capacitor
$C_0 = 1/A_0$ in series with $n$ parallel $LC$ branches whose
inductances and capacitances are given by
\[
   L_k = \frac{A_k}{B_k^2}, \qquad
   C_k = \frac{1}{A_k}, \qquad k=1,\dots,n.
\]

Each branch represents one symmetric pair of interpolation nodes $(z_k,-\bar z_k)$ with parameters
$A_k=2b_k$ and $B_k=\sqrt{a_k^2+b_k^2}$ (see~\eqref{e:ABformulaA}), while the term $A_0/z$ arises from the collection of purely imaginary nodes $z_j=i b_j$.

The symmetry of the interpolation data ensures reciprocity and
lossless energy exchange between the inductive and capacitive components,
whereas any deviation from symmetry ($a_k\neq0$) introduces controlled
dissipation quantified by the c-entropy and dissipation coefficient.
Hence, each interpolation pair contributes a resonant $LC$ channel whose
strength and resonance frequency are determined by $A_k$ and $B_k$, respectively.
The analytic model $V_\Delta(z)$ thus provides a direct bridge between the
geometry of the interpolation data and the physical theory of passive
positive-real networks, linking the operator-theoretic structure of the
interpolation L-system~$\Theta_\Delta$ with the classical
$LC$-network synthesis framework
(see, e.g.,~\cite{EfimPotap73}).
\end{remark}

\begin{figure}[h!]
    \centering
    \includegraphics[width=0.8\linewidth]{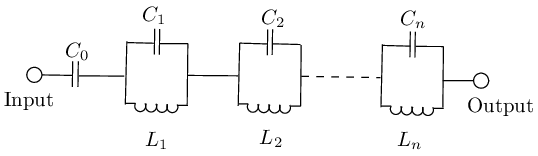}
    \caption{Equivalent electrical network corresponding to the impedance function
    $Z(p)=\frac{A_0}{p}+\sum_{k=1}^n A_k\,\frac{p}{B_k^2+p^2}$.}
    \label{fig:LCnetwork}
\end{figure}

The correspondence between the analytic impedance function $V_\Delta(z)$ and its physical network realization is illustrated schematically in Figure~\ref{fig:LCnetwork}. Each pair of symmetric interpolation nodes $\{z_k, -\bar z_k\}$ gives rise to one resonant $LC$ branch of the circuit. The network consists of $n$ parallel $LC$ branches having inductances $L_k = A_k/B_k^2$ and capacitances $C_k = 1/A_k$ for $k=1,\dots,n$. Each branch corresponds to a symmetric pair $(z_k,-\bar z_k)$ with parameters $A_k=2b_k$ and $B_k^2=a_k^2+b_k^2$, which determine the branch strength and resonance frequency. In addition, the single capacitor $C_0 = 1/A_0$ shown in Figure~\ref{fig:LCnetwork} represents the collective contribution of all purely imaginary nodes, providing the baseline capacitive channel responsible for static energy storage and overall charge balance in the system.
Together these elements form the complete passive one-port network associated with the model impedance function $V_\Delta(z)$.

\smallskip
\noindent
The representation~\eqref{e:Vgeneral} and its network realization shown in Figure~\ref{fig:LCnetwork} illustrate how the analytic geometry of the interpolation nodes governs the energy distribution within the corresponding L-system.
Each symmetric pair $(z_k,-\bar z_k)$ defines a resonant channel of reversible energy exchange, while the purely imaginary nodes produce the static component responsible for baseline capacitive storage. Together these elements form a balanced structure whose entropy and dissipation characteristics are determined entirely by the placement of
the interpolation data. This correspondence establishes a direct bridge between the analytic geometry of symmetric interpolation configurations and the physical mechanisms of energy flow in the associated L-system.

\medskip
\noindent
\textbf{Examples.}
To conclude this section, we illustrate the analytic and physical structures developed above by evaluating the model impedance function for concrete interpolation data. The first example treats the simplest nontrivial symmetric configuration, while the second adds a purely imaginary node to the same symmetric pair,
demonstrating the combined resonant-capacitive character of the resulting system and its influence on the associated entropy and dissipation invariants.

\begin{example}[Symmetric case: finite c-entropy]\label{ex:num7}
Consider the symmetric interpolation data consisting of two nodes
\[
z_1 = 1 + 2i, \qquad z_2 = -1 + 2i,
\]
so that $a_1 = 1$ and $b_1 = 2$.
According to Lemma \ref{l:symm-pair-cayley}, the corresponding parameters of the impedance function are
\[
A = 2b = 4,
\qquad
B = \sqrt{a^2 + b^2} = \sqrt{5}.
\]
Hence, the impedance function of the interpolation L-system $\Theta_\Delta$ is
\begin{equation}\label{e:exV}
V_{\Theta_\Delta}(z) = \frac{4z}{5 - z^2}.
\end{equation}
A quick check shows that $V_{\Theta_\Delta}(\pm 1 + 2i) = i$, confirming the interpolation property.
Moreover, for $z = x + i y$ with $y>0$, the imaginary part is
\[
\IM V_{\Theta_\Delta}(z)
= \frac{4y(5+x^2+y^2)}{(5-x^2-y^2)^2 + 4x^2y^2} > 0,
\]
so that $V_{\Theta_\Delta}(z)$ is indeed a Herglotz-Nevanlinna function.
Evaluating at $z=i$ gives $V_{\Theta_\Delta}(i) = \frac{2}{3}i$,
placing $V_{\Theta_\Delta}(z)$ in the bounded Donoghue class
$\widehat{\sM_\kappa}$ with $\kappa = \frac{1}{5}$ (see~\cite[Section~3]{BT-23}).

The corresponding transfer function, obtained via the Cayley transform, is
\[
W_{\Theta_\Delta}(z)
= \frac{1 - iV_{\Theta_\Delta}(z)}{1 + iV_{\Theta_\Delta}(z)}
   = \frac{3 - z^2 - 4i z}{3 - z^2 + 4i z},
\]
which coincides with the product representation
\[
W_{\Theta_\Delta}(z)
= \frac{(z - \bar z_1)(z + z_1)}{(z - z_1)(z + \bar z_1)}
   = \frac{(z - 1 + 2i)(z + 1 + 2i)}{(z - 1 - 2i)(z + 1 - 2i)}.
\]
This confirms perfect agreement with the general formula~\eqref{e-28=W}.

Finally, the c-entropy and dissipation coefficient computed from
\eqref{e:exV} via formulas~\eqref{e-69-ent-dis} yield
\[
\calS = \ln 5 \approx 1.6094, \qquad
\calD = 1 - e^{-2\calS} = \frac{24}{25} = 0.96,
\]
demonstrating that the symmetric pair of nodes produces an
interpolation L-system of finite c-entropy corresponding to the
bounded Donoghue class $\widehat{\sM_{\frac{1}{5}}}$.
\end{example}

\begin{example}[Mixed configuration: symmetric pair plus one purely imaginary node]\label{ex:num8}
Consider the interpolation data
\[
   z_1 = 1 + 2i,\qquad
   z_2 = -1 + 2i,\qquad
   z_3 = 3i .
\]
By Corollary~\ref{c:pair-plus-imag} (with $a=1$, $b=2$, $c=3$), we have
\[
   B^2 = a^2+b^2 = 5,\;
   \widetilde B^{\,2} = B^2 + 2bc = 17,\;
   C = \frac{cB^2}{\widetilde B^{\,2}} = \frac{15}{17},\;
   A = (2b+c) - C = \frac{104}{17},
\]
and hence the (canonical) impedance function is
\[
   V_{\Theta_\Delta}(z)
   = -\frac{C}{z} \;+\; A\,\frac{z}{\widetilde B^{\,2}-z^2}
   = -\frac{15}{17}\,\frac{1}{z} \;+\; \frac{104}{17}\,\frac{z}{17 - z^2}.
\]
A direct substitution shows that
\[
   V_{\Theta_\Delta}(1+2i)=V_{\Theta_\Delta}(-1+2i)=V_{\Theta_\Delta}(3i)=i,
\]
confirming the interpolation property.

The transfer function obtained via the Cayley transform is
\[
   W_{\Theta_\Delta}(z)=\frac{1 - iV_{\Theta_\Delta}(z)}{1 + iV_{\Theta_\Delta}(z)}.
\]
Evaluating at $z=-i$ yields
\[
   V_{\Theta_\Delta}(-i) \;=\; -\frac{11}{9}\,i,
   \qquad
   W_{\Theta_\Delta}(-i) \;=\; -\frac{1}{10},
   \qquad
   \bigl|W_{\Theta_\Delta}(-i)\bigr| \;=\; \frac{1}{10}.
\]
Therefore the c-entropy and dissipation coefficient are
\[
   \mathcal{S} \;=\; -\ln\!\bigl|W_{\Theta_\Delta}(-i)\bigr|
   = \ln 10 \;\approx\; 2.3026,
   \qquad
   \mathcal{D} \;=\; 1 - e^{-2\mathcal{S}}
   = 1 - 10^{-2} \;=\; 0.99.
\]
Compared with the purely symmetric case of Example~\ref{ex:num7}
(where $\mathcal{S}=\ln 5$ and $\mathcal{D}=24/25$), the additional purely
imaginary node increases both the c-entropy and the dissipation.
\end{example}

\medskip
Having established the analytic and physical structure of symmetric interpolation systems,
we now turn to the explicit \emph{operator realization} of the model $\Theta_\Delta$,
where these abstract constructions are implemented concretely in terms of
state-space operators and their spectral properties.


\section{Operator Realization of the Model $\Theta_\Delta$}\label{s8}

The interpolation model $\Theta_\Delta$ occupies a central position in our analysis,
serving as the bridge between the analytic characterization of interpolation data and
the operator-theoretic realization of their dissipative structure.

Building on the classification of entropy regimes obtained in Section~\ref{s6},
we now turn to the operator model $\Theta_\Delta$ introduced in Section~\ref{s3}
and examine its concrete realization and invariant properties.
Below we present a complete worked example that connects all main constructions
introduced in Sections~\ref{s3}-\ref{s5}.

Consider the interpolation data
\[
   z_1=2i,\qquad z_2=3i,\qquad v_1=v_2=i,
\]
and explicitly build the corresponding Pick-form interpolation system $\vec\Theta$,
the canonical model $\Theta_\Delta$, and the coupled one-dimensional multiplication model
$\Theta=\Theta_1\!\cdot\!\Theta_2$.
We then compute their transfer functions, verify that they coincide, and evaluate
the $c$-entropy and dissipation coefficients to confirm the additivity and invariance
properties established earlier.

\subsection{Setup of the interpolation data}\label{ss-81}

The Nevanlinna-Pick data $(z_k,v_k)\in\dC_+\times\dC_+$ for $k=1,2$
determine the Pick matrices from \eqref{PQ}
\[
  \mathbb{P}_{j,k}=\frac{v_k-\bar v_j}{z_k-\bar z_j}, \qquad
  \mathbb{Q}_{j,k}=\frac{z_kv_k-\bar z_j\bar v_j}{z_k-\bar z_j},
\]
which are strictly positive in this example.
All subsequent constructions will use these matrices
to generate self-adjoint, dissipative, and coupled system realizations.

\subsection{Pick-form interpolation system $\vec\Theta$}\label{ss-82}
We begin with the construction of the interpolation L-system $\vec{\Theta}$ in the Pick form,
directly corresponding to the given interpolation data $(z_k,v_k)$.
This realization serves as the analytic starting point for comparing transfer and impedance
characteristics across the different system representations considered below.
For \(z_1=2i,\,z_2=3i\) and \(v_1=v_2=i\) (so \(\bar v_1=\bar v_2=-i\)), the denominators are
\[
  z_1-\bar z_1=4i,\;\; z_2-\bar z_2=6i,\;\; z_1-\bar z_2=5i,\;\; z_2-\bar z_1=5i,
\]
and \(v_k-\bar v_j=2i\) for all \(j,k\). Hence
\[
  \mathbb{P}
  =\begin{pmatrix}
      \dfrac{2i}{4i} & \dfrac{2i}{5i}\\[6pt]
      \dfrac{2i}{5i} & \dfrac{2i}{6i}
    \end{pmatrix}
  =\begin{pmatrix}
      \tfrac12 & \tfrac{2}{5}\\[3pt]
      \tfrac{2}{5} & \tfrac13
    \end{pmatrix},
  \qquad
  \det\mathbb{P}=\frac{1}{150}>0.
\]
For \(\mathbb Q\), note \(z_k v_k= -\IM z_k\) and \(\bar z_j\,\bar v_j= -\IM z_j\), so the numerators are
\(\IM z_j-\IM z_k\). Thus
\[
  \mathbb{Q}
  =\begin{pmatrix}
      0 & \dfrac{2-3}{\,5i\,}\\[6pt]
      \dfrac{3-2}{\,5i\,} & 0
    \end{pmatrix}
  =\begin{pmatrix}
      0 & \dfrac{i}{5}\\[3pt]
      -\dfrac{i}{5} & 0
    \end{pmatrix}.
\]

{Let us find the state space and inner product.}
We work in \(\dC^2\) with the weighted inner product \((\xi,\eta)_{\dC^2_{\mathbb P}}=\eta^*\mathbb P\,\xi\).
Let
\[
  \vec A = \mathbb P^{-1}\mathbb Q,
  \qquad
  \varphi=\begin{pmatrix} v_1\\  v_2\end{pmatrix}=
          \begin{pmatrix}i\\ i\end{pmatrix},
  \qquad
  \vec g = \mathbb P^{-1}\varphi.
\]
Since
\(
  \mathbb P^{-1}
  =
  \begin{pmatrix}
   50 & -60\\
   -60 & 75
  \end{pmatrix},
\)
we get
\[
  \vec g
  = \begin{pmatrix}50 & -60\\ -60 & 75\end{pmatrix}
    \begin{pmatrix}i\\ i\end{pmatrix}
  = \begin{pmatrix}-10i\\[3pt] 15i\end{pmatrix},
  \qquad
  \vec A
  = \mathbb P^{-1}\mathbb Q
  = \begin{pmatrix}
      12i & 10i\\[2pt]
      -15i & -12i
    \end{pmatrix}.
\]
One readily checks \(\vec A\) is self-adjoint with respect to \((\cdot,\cdot)_{\dC^2_{\mathbb P}}\)
since \(\mathbb Q=\mathbb Q^*\).

Define
\[
  \vec T = \vec A + i(\cdot,\vec g)_{\dC^2_{\mathbb P}}\,\vec g
          \;=\; \vec A + i\,\vec g\,(\vec g^*\mathbb P),
\qquad
  \vec K c = c\,\vec g,\;\; c\in\dC.
\]
Because \(\vec g^*\mathbb P=[\,-i\;\; -i\,]\), the rank-one part is
\[
  i\,\vec g\,(\vec g^*\mathbb P)
  = i\begin{pmatrix}-10i\\ 15i\end{pmatrix}\![\,-i\;\; -i\,]
  = \begin{pmatrix}-10i & -10i\\[2pt] 15i & 15i\end{pmatrix}.
\]
Therefore
\[
  \vec T
  = \begin{pmatrix}
      12i & 10i\\[1pt]
      -15i & -12i
    \end{pmatrix}
    +
    \begin{pmatrix}-10i & -10i\\[1pt] 15i & 15i\end{pmatrix}
  = \begin{pmatrix} 2i & 0\\[1pt] 0 & 3i \end{pmatrix}.
\]
Thus
\[
\vec T=\begin{pmatrix}2i&0\\[1pt]0&3i\end{pmatrix}.
\]
By construction,
\[
\vec T=\vec A+i(\cdot,\vec g)_{\mathbb C^2_{\mathbb P}}\vec g,
\]
so that $\RE \vec T=\vec A$ and $\Im \vec T=\vec K\vec K^*$ in the Hilbert space
$\mathbb C^2_{\mathbb P}$. In particular, $\vec T$ is dissipative.

We have constructed the interpolation system \(\vec\Theta\) that is the Liv\v{s}ic canonical scattering system in the Pick form
\[
  \vec\Theta
  = \begin{pmatrix}
      \vec T & \vec K & 1\\
      \dC^2_{\mathbb P} & & \dC
    \end{pmatrix}
  \quad\text{with}\quad
  \vec T=\begin{pmatrix}2i&0\\[1pt]0&3i\end{pmatrix},\;\;
  \vec K c = c\begin{pmatrix}10i\\[1pt]-15i\end{pmatrix}.
\]

As in \eqref{50}, the impedance function is
\[
  V_{\vec\Theta}(z)=\vec K^*(\vec A - zI)^{-1}\vec K
  = (\cdot,\vec g)_{\dC^2_{\mathbb P}}\big((\vec A - zI)^{-1}\vec g\big).
\]
From \((\vec A - z_k I)e_k=\vec g\) (which holds here for \(k=1,2\)), we obtain
\[
  V_{\vec\Theta}(z_k)=(e_k,\vec g)_{\dC^2_{\mathbb P}}
  = \vec g^*\mathbb P\,e_k
  = i, \qquad k=1,2.
\]
Thus \(\vec\Theta\) solves the Nevanlinna-Pick interpolation problem for the data
\((z_1,z_2)=(2i,3i)\) and \((v_1,v_2)=(i,i)\). The transfer function is the Blaschke product
\[
  W_{\vec\Theta}(z)
  = \prod_{k=1}^2\frac{z-\bar z_k}{z-z_k}
  = \frac{z+2i}{z-2i}\cdot\frac{z+3i}{z-3i},
\]
and the corresponding impedance function equals
\(
  V_{\vec\Theta}(z)= i\,\dfrac{W_{\vec\Theta}(z)-1}{W_{\vec\Theta}(z)+1}
\),
with \(V_{\vec\Theta}(2i)=V_{\vec\Theta}(3i)=i\).

\subsection{Canonical model system $\Theta_\Delta$}\label{ss-83}
Next, we construct the canonical model system $\Theta_\Delta$ associated with the same
interpolation data.  This model provides a concrete operator realization that is unitarily
equivalent to $\vec{\Theta}$, sharing the same transfer function while offering a more
transparent matrix representation of its internal structure.

For the model L-system $\Delta$ in Section \ref{s3} (see \eqref{99}-\eqref{e-27-Theta}), with $m=2$ and
$z_1=2i$ (so $\IM z_1=2$), $z_2=3i$ (so $\IM z_2=3$), we have
\[
  T_\Delta
  = \begin{pmatrix}
      z_1 & 2i\,\sqrt{\IM z_1}\,\sqrt{\IM z_2}\\[3pt]
      0   & z_2
    \end{pmatrix}
  = \begin{pmatrix}
      2i & 2i \sqrt{2}\sqrt{3}\\[3pt]
      0  & 3i
    \end{pmatrix}
  = \begin{pmatrix}
      2i & 2\sqrt{6}\,i\\[3pt]
      0  & 3i
    \end{pmatrix},
\]
and
\[
  g_\Delta
  = \begin{pmatrix}\sqrt{\IM z_1}\\ \\ \sqrt{\IM z_2}\end{pmatrix}
  = \begin{pmatrix}\sqrt{2}\\ \\ \sqrt{3}\end{pmatrix}.
\]
Define $K_\Delta(c)=c\,g_\Delta$. Then $\IM T_\Delta=(\cdot,g_\Delta)\,g_\Delta$ and
\[
  \Theta_\Delta
  = \begin{pmatrix}
      T_\Delta & K_\Delta & 1\\
      \dC^2 & & \dC
    \end{pmatrix}
\]
is the Liv\v{s}ic scattering system in the model (Pick) form.

Let's find the transfer function $W_{\Theta_\Delta}(z)$. By the general formula in Section~\ref{s3} (see \eqref{e-28=W}),
\[
  W_{\Theta_\Delta}(z)
  = 1 - 2i\,( (T_\Delta - zI)^{-1}g_\Delta,\,g_\Delta )
  = \prod_{k=1}^2 \frac{z-\bar z_k}{z-z_k}
  = \frac{z+2i}{z-2i}\cdot\frac{z+3i}{z-3i}.
\]
Comparing with subsection \ref{ss-82}, we found for the Pick-form realization \(\vec\Theta\) that
\[
  W_{\vec\Theta}(z)
  = \frac{z+2i}{z-2i}\cdot\frac{z+3i}{z-3i}.
\]
Hence, as expected,
\[
  W_{\Theta_\Delta}(z)\equiv W_{\vec\Theta}(z),
\]
so the model system $\Theta_\Delta$ and the Pick-form realization $\vec\Theta$ have the same
transfer function (and therefore the same impedance function
\(V=i\frac{W-1}{W+1}\), with \(V(2i)=V(3i)=i\)).
Since both realizations are minimal (see Section~\ref{s3}) and have identical transfer functions, the canonical model and the Pick-form realization are unitarily equivalent.

\subsection{Coupled multiplication-operator system $\Theta$}\label{ss-84}
Finally, we consider two one-dimensional multiplication-operator L-systems $\Theta_1$ and $\Theta_2$. By Lemma~\ref{lem:coupling-minimal}, their coupling $\Theta=\Theta_1\!\cdot\!\Theta_2$ is minimal. Since it has the same transfer function as $\Theta_\Delta$, the two systems are unitarily equivalent.

This construction illustrates how the abstract coupling framework reproduces the
same transfer behavior as the interpolation and model systems, thereby confirming
the invariance of the associated c-entropy and dissipation coefficient.

We work on the one-dimensional state space $\calH=\dC$ with the normalized basis vector $h_0=1$.
For a multiplication-operator L-system with parameter $\lambda_0\in\dC_+$ we take
\[
   T h=\lambda_0 h,\qquad
   Kc=(\sqrt{\IM\lambda_0}\,c)\,h_0=\sqrt{\IM\lambda_0}\,c,\qquad c\in\dC.
\]
Then $K^*h=(h,\sqrt{\IM\lambda_0})=\sqrt{\IM\lambda_0}\,h$, so that $KK^*=(\IM\lambda_0)I$ and
$\IM T = (\IM\lambda_0)I = KK^*$.

Based on our data set we construct two one-dimensional multiplication L-systems. For $z_1=2i$ and $z_2=3i$ we set
\[
    \begin{aligned}
   \Theta_1&=
   \begin{pmatrix}
      T_1 & K_1 & 1\\
      \dC   &    & \dC
   \end{pmatrix},\quad
   T_1( h ) = 2i\, h,\quad K_1 c = \sqrt{2}\,c,
   \\
   \Theta_2&=
   \begin{pmatrix}
      T_2 & K_2 & 1\\
      \dC   &    & \dC
   \end{pmatrix},\quad
   T_2( h ) = 3i\, h,\quad K_2 c = \sqrt{3}\,c.
   \end{aligned}
\]
Their transfer functions (cf.\ Section~4, \eqref{mult-W}) are
\[
   W_{\Theta_1}(z)=\frac{\bar{2i}-z}{2i-z}=\frac{z+2i}{z-2i},
   \qquad
   W_{\Theta_2}(z)=\frac{\bar{3i}-z}{3i-z}=\frac{z+3i}{z-3i}.
\]

Following Section~4 (see \eqref{coupling}-\eqref{T-K-coupling}), the \emph{coupled} system
$\Theta=\Theta_1\cdot\Theta_2$ acts on $\calH\oplus\calH=\dC^2$ and is defined by
\[
   \Theta=
   \begin{pmatrix}
      \mathbf{T} & \mathbf{K} & 1\\
      \dC^2       &            & \dC
   \end{pmatrix},
   \qquad
   \mathbf{T}=
   \begin{pmatrix}
      T_1 & 2i K_1 K_2^*\\[2pt]
      0   & T_2
   \end{pmatrix},
   \qquad
   \mathbf{K}=\begin{pmatrix}K_1\\ K_2\end{pmatrix}.
\]
Since $K_1K_2^*$ is multiplication by $\sqrt{2}\,\sqrt{3}=\sqrt{6}$ on $\dC$, we get
\[
   \mathbf{T}=
   \begin{pmatrix}
      2i & 2i\sqrt{6}\\[2pt]
      0  & 3i
   \end{pmatrix},
   \qquad
   \mathbf{K} c = \begin{pmatrix}\sqrt{2}\\ \sqrt{3}\end{pmatrix}c,\quad c\in\dC.
\]
A direct computation gives
\[
   \IM \mathbf{T}
   = \frac{1}{2i}(\mathbf{T}-\mathbf{T}^*)
   = \begin{pmatrix} 2 & 0\\ 0 & 3\end{pmatrix}
   = \mathbf{K}\,\mathbf{K}^*,
\]
so $\Theta$ is again a canonical dissipative L-system.

By multiplicativity of transfer functions for the coupling (Section~\ref{s4}),
\[
   W_{\Theta}(z)=W_{\Theta_1}(z)\,W_{\Theta_2}(z)
   = \frac{z+2i}{z-2i}\cdot\frac{z+3i}{z-3i}.
\]
This matches exactly the transfer function obtained in Step~1 for the Pick-form realization
$\vec\Theta$ and in Step~2 for the model system $\Theta_\Delta$:
\[
   W_{\Theta}(z) \equiv W_{\vec\Theta}(z) \equiv W_{\Theta_\Delta}(z)
   = \frac{z+2i}{z-2i}\cdot\frac{z+3i}{z-3i}.
\]
Consequently, all three realizations are (unitarily) equivalent at the level of transfer data,
and thus have the same impedance function $V(z)=i\frac{W(z)-1}{W(z)+1}$ and the same invariants
$c$-entropy and dissipation coefficient.

\subsection{c-Entropy and dissipation coefficients}\label{ss-85}
Having established the three equivalent realizations---the interpolation system
$\vec{\Theta}$, the model system $\Theta_\Delta$, and the coupled system
$\Theta=\Theta_1\!\cdot\!\Theta_2$---we now turn to the computation of their intrinsic
quantitative invariants.  Specifically, we evaluate the c-entropy $\mathcal{S}$ and the
dissipation coefficient $\mathcal{D}$ for each realization and verify the additivity and
composition laws derived earlier.  The results confirm that these quantities are invariant
under unitary equivalence and depend solely on the geometric placement of the interpolation
nodes, thus providing a concrete illustration of the intrinsic nature of entropy and
dissipation in the theory of L-systems.

We collect the invariants for:
\[
\Theta_1\ \ (\lambda_0=2i),\qquad
\Theta_2\ \ (\lambda_0=3i),\qquad
\Theta=\Theta_1\cdot\Theta_2,\qquad
\vec\Theta,\qquad
\Theta_\Delta.
\]

For a one-dimensional multiplication model (Section~\ref{s4}) the formulas are
\[
   \cS(\lambda_0)=\tfrac{1}{2}\ln\!\frac{(\RE\lambda_0)^2+(1+\IM\lambda_0)^2}
                                     {(\RE\lambda_0)^2+(1-\IM\lambda_0)^2},
   \qquad
   \cD(\lambda_0)=\frac{4\,\IM\lambda_0}{(\RE\lambda_0)^2+(1+\IM\lambda_0)^2}.
\]
With $\lambda_0=2i$ and $\mu_0=3i$ we obtain
\[
\begin{aligned}
  \cS_1&=\cS(\Theta_1)=\tfrac{1}{2}\ln\frac{(1+2)^2}{(1-2)^2}
         \;=\;\ln\!\left|\frac{1+2}{1-2}\right|
         \;=\;\ln 3,
  &\qquad \cD_1&=\frac{4\cdot 2}{(1+2)^2}=\frac{8}{9},\\[2mm]
  \cS_2&=\cS(\Theta_2)=\tfrac{1}{2}\ln\frac{(1+3)^2}{(1-3)^2}
         \;=\;\ln\!\left|\frac{1+3}{1-3}\right|
         \;=\;\ln 2,
  &\qquad \cD_2&=\frac{4\cdot 3}{(1+3)^2}=\frac{12}{16}=\frac{3}{4}.
\end{aligned}
\]

Additivity of c-entropy and the dissipation coupling law give
\[
   \cS(\Theta)=\cS_1+\cS_2=\ln 3+\ln 2=\ln 6,
   \qquad
   1-\cD(\Theta)=(1-\cD_1)(1-\cD_2),
\]
hence
\[
   \cD(\Theta)
   \;=\;\cD_1+\cD_2-\cD_1\cD_2
   \;=\;\frac{8}{9}+\frac{3}{4}-\frac{8}{9}\cdot\frac{3}{4}
   \;=\;\frac{35}{36}.
\]

For Pick-form L-system $\vec\Theta$ and model system $\Theta_\Delta$ from the above we have
\[
   W_{\vec\Theta}(z)\equiv W_{\Theta_\Delta}(z)
   \equiv W_{\Theta}(z)
   \;=\;\frac{z+2i}{z-2i}\cdot\frac{z+3i}{z-3i}.
\]
Therefore all three realizations have the same invariants:
\[
    \begin{aligned}
   \cS(\vec\Theta)&=\cS(\Theta_\Delta)=\cS(\Theta)
   \;=\;-\ln\!\big|W_\Theta(-i)\big|
   \;=\;-\ln\!\left|\frac{-i+2i}{-i-2i}\cdot\frac{-i+3i}{-i-3i}\right|\\
   \;&=\;-\ln\!\left|\frac{i}{\,-3i\,}\cdot\frac{2i}{\,-4i\,}\right|
   \;=\;-\ln\!\left|\frac{1}{6}\right|
   \;=\;\ln 6,
   \end{aligned}
\]
and, using $1-\cD=\prod_k (1-\cD_k)$ (or directly from $W$),
\[
   \cD(\vec\Theta)=\cD(\Theta_\Delta)=\cD(\Theta)=\frac{35}{36}.
\]
We can summarize all this data in Table \ref{t-1}.

\paragraph{\textbf{Numerical values.}}
\[
   \cS_1=\ln 3\approx 1.098612,\;
   \cS_2=\ln 2\approx 0.693147,\;
   \cS(\text{any of }\Theta,\vec\Theta,\Theta_\Delta)=\ln 6\approx 1.791759,
\]
\[
   \cD_1=\frac{8}{9}\approx 0.8888889,\quad
   \cD_2=\frac{3}{4}=0.75,\quad
   \cD(\text{any of }\Theta,\vec\Theta,\Theta_\Delta)=\frac{35}{36}\approx 0.9722222.
\]

\begin{table}[h!]\label{t-1}
\centering
\renewcommand{\arraystretch}{1.2}
\begin{tabular}{lcc}
\toprule
L-system & $c$-Entropy $\cS$ & Dissipation $\cD$ \\ \midrule
$\Theta_1$ ($2i$) & $\ln 3$ & $8/9$ \\
$\Theta_2$ ($3i$) & $\ln 2$ & $3/4$ \\
Coupling $\Theta=\Theta_1\!\cdot\!\Theta_2$ & $\ln 6$ & $35/36$ \\
Pick-form system $\vec\Theta$ & $\ln 6$ & $35/36$ \\
Model system $\Theta_\Delta$ & $\ln 6$ & $35/36$ \\ \bottomrule \\
\end{tabular}
\caption{Invariants of all realizations for the two-node interpolation data.}
\end{table}

Note that all constructions yield the same transfer function
$$
W(z)=\frac{z+2i}{z-2i}\cdot\frac{z+3i}{z-3i},
$$
and consequently the same invariants.
This confirms that both $c$-entropy and the dissipation coefficient are
\emph{intrinsic invariants} of the interpolation data.
The example numerically verifies the additivity law
\(
\cS=\cS_1+\cS_2
\)
and the nonlinear coupling law
\(
\cD=\cD_1+\cD_2-\cD_1\cD_2,
\)
demonstrating consistency between the Pick-form, canonical, and coupled formulations
of L-systems.

In addition, we see that although $\Theta_\Delta$ is not literally the coupling of the one-node models, it is unitarily equivalent to such a coupling, since both realizations are minimal and share the same transfer function. Since $\cS$ and $\cD$ are defined via $|W(-i)|$ and are invariant under unitary equivalence,
their values are \emph{intrinsic} to the interpolation data. The additivity of $\cS$
and the nonlinear law $\cD=\cD_1+\cD_2-\cD_1\cD_2$ are verified exactly in this example.

In conclusion, the explicit constructions and computations presented in this section confirm
that the interpolation system $\vec{\Theta}$, the model system $\Theta_\Delta$, and the
coupled multiplication system $\Theta=\Theta_1\!\cdot\!\Theta_2$ are equivalent not only in
their transfer and impedance characteristics but also in their intrinsic invariants.  The
calculated values of the c-entropy and dissipation coefficient coincide across all
realizations, providing concrete numerical evidence for the theoretical results of
Sections~\ref{s6} and~\ref{s8}.  These findings demonstrate that the invariants
$\mathcal{S}$ and $\mathcal{D}$ depend exclusively on the geometric placement of the
interpolation nodes in $\dC_+$, with maximal finite c-entropy achieved for purely imaginary
nodes.  Thus, the examples substantiate the intrinsic character of these quantities and
establish a direct correspondence between the analytic structure of the interpolation data
and the dynamical properties of the associated L-systems.

\section{Concluding Remarks}\label{s9}

In this paper, we followed a unified framework linking the Nevanlinna-Pick interpolation problem with the theory of L-systems that was developed in \cite{ABT}.  Explicit formulas were derived for the c-entropy and dissipation coefficient, together with precise conditions under which the c-entropy is finite, maximal, or infinite.  These quantities were shown to be intrinsic invariants determined solely by the geometric placement of interpolation nodes in~$\dC_+$.

The interpolation model~$\Theta_\Delta$ constructed here provides an operator realization that reproduces the analytic data of the interpolation problem while preserving its internal dissipative structure.  The accompanying examples confirm that the same transfer behavior and invariant measures of c-entropy and dissipation are shared by the Pick-form, canonical, and coupled realizations.

Special attention was given to symmetric interpolation configurations, for which the
impedance function admits an explicit rational representation and a direct physical
interpretation in terms of equivalent $LC$-networks.  This correspondence illustrates how
the geometric symmetry of interpolation data translates into balanced energy exchange
within the associated L-system.

Overall, the results demonstrate that the c-entropy and dissipation coefficient provide
effective quantitative bridges between analytic interpolation data and operator-theoretic
dissipation models, revealing a concrete link between system geometry, entropy, and
physical energy balance.

\bigskip


\textbf{Data sharing.}  Data sharing is not applicable to this article as no datasets were generated or analyzed during the current study.

\textbf{Funding declaration.}  No funding was received for conducting this study or to assist with the preparation of this manuscript.

\appendix
\section{The Multiplication Theorem }\label{A}
\begin{theorem}[cf. \cite{Bro}]\label{unitar}
 Assume that $\Theta$ is the  L-system (operator) coupling of the form \eqref{coupling}-\eqref{diss-coupling-general} referred to in Section \ref{s4}.

 Then \begin{equation}\label{e-91-mult}
W_\Theta(z)=W_{\Theta_1}(z)\cdot W_{\Theta_2}(z),\quad
z\in\rho(T_1)\cap\rho(T_2).
\end{equation}
\end{theorem}
\begin{proof} Suppose, for definiteness,
that    $T_1$ and $T_2$ are  multiplication operators in $\calH$ of the form \eqref{mult-T} given by
 \begin{equation*}
    T_1 h=\lambda_0 h\quad \textrm{and}\quad T_2 h=\mu_0 h,\quad h\in\calH.
\end{equation*}
Since
\begin{equation*}
  \begin{aligned}
   \mathbf{T}\left[
               \begin{array}{c}
                 h_1 \\
                 h_2 \\
               \end{array}
             \right]&=\left(
                       \begin{array}{cc}
                         T_1 & 2iK_1K_2^* \\
                         0 & T_2 \\
                       \end{array}
                     \right)\left[ \begin{array}{c}
                 h_1 \\
                 h_2 \\
               \end{array}
             \right]=\left[
               \begin{array}{c}
                 T_1 h_1+2i K_1K_2^*h_2 \\
                 T_2 h_2 \\
               \end{array}
             \right]\\
             &=\left[
               \begin{array}{c}
                 \lambda_0 h_1+2i  \sqrt{(\IM\lambda_0)\cdot(\IM\mu_0)}\,h_2\\
                 \mu_0 h_2 \\
               \end{array}
             \right],\quad\forall h_1,h_2\in\calH.
      \end{aligned}
   \end{equation*}
using elementary calculations one finds
\begin{equation*}
    \begin{aligned}
(\mathbf{T}-zI)^{-1}&=\left(
                       \begin{array}{cc}
                         \lambda_0-z & 2i\sqrt{\IM\lambda_0\cdot\IM\mu_0} \\
                         0 & \mu_0-z \\
                       \end{array}
                     \right)^{-1}\\
                     &=\frac{1}{(\lambda_0-z)(\mu_0-z)}\left(
                       \begin{array}{cc}
                         \mu_0-z & -2i\sqrt{\IM\lambda_0\cdot\IM\mu_0} \\
                         0 & \lambda_0-z \\
                       \end{array}
                     \right).
    \end{aligned}
\end{equation*}
Taking into account that
$$
2i\sqrt{\IM\lambda_0\cdot\IM\mu_0}=\sqrt{(\lambda_0-\bar\lambda_0)(\mu_0-\bar\mu_0)},
$$
and using representations
\begin{equation*}
    \mathbf{K}c=\left[
               \begin{array}{c}
                 K_1 c \\
                 K_2 c\\
               \end{array}
             \right]=\left[
               \begin{array}{c}
                 (\sqrt{\IM\lambda_0}\, c) h_0 \\
                (\sqrt{\IM\mu_0}\, c) h_0\\
               \end{array}
             \right].
\end{equation*}
and
\begin{equation*}
    \mathbf{K}^*\left[
               \begin{array}{c}
                 h_1 \\
                 h_2 \\
               \end{array}
             \right]=K_1^*h_1+K_2^*h_2=(h_1, \sqrt{\IM\lambda_0}\,h_0)+(h_2, \sqrt{\IM\mu_0}\,h_0).
\end{equation*}
we find
$$
    \begin{aligned}
2i\mathbf{K}^*(\mathbf{T}-zI)^{-1}\mathbf{K}&=\frac{2i\left((\mu_0-z)\IM\lambda_0-2i\IM\lambda_0\IM\mu_0+(\lambda_0-z)\IM \mu_0 \right)}{(\lambda_0-z)(\mu_0-z)}\\
&=\frac{(\lambda_0-\bar\lambda_0)(\bar\mu_0-z)+(\lambda_0-z)(\mu_0-\bar\mu_0)}{(\lambda_0-z)(\mu_0-z)}.
 \end{aligned}
 $$
 Finally,
 $$
    \begin{aligned}
W_\Theta(z)&=1-2i\mathbf{K}^*(\mathbf{T}-zI)^{-1}\mathbf{K}
=1-\frac{(\lambda_0-\bar\lambda_0)(\bar\mu_0-z)+(\lambda_0-z)(\mu_0-\bar\mu_0)}{(\lambda_0-z)(\mu_0-z)}\\
&=\frac{(\lambda_0-z)(\mu_0-z)-(\lambda_0-\bar\lambda_0)(\bar\mu_0-z)-(\lambda_0-z)(\mu_0-\bar\mu_0)}{(\lambda_0-z)(\mu_0-z)}\\
&=\frac{(\bar\lambda_0-z)(\bar\mu_0-z)}{(\lambda_0-z)(\mu_0-z)}=\frac{\bar\lambda_0-z}{\lambda_0-z}\cdot\frac{\bar\mu_0-z}{\mu_0-z}
=W_{\Theta_1}(z)\cdot W_{\Theta_2}(z),
 \end{aligned}
 $$
 completing the proof. Here we have used eq. \eqref{mult-W}.
\end{proof}


\end{document}